\documentclass{amsart}
\usepackage{amssymb,
enumerate,
hyperref,
makecell
}
\hypersetup{colorlinks=true}
\numberwithin{equation}{section}
\usepackage[T1]{fontenc}

\newcommand{\U}{\mathcal{U}}
\newcommand{\V}{\mathcal{M}}
\newcommand{\D}{\mathcal{D}}
\newcommand{\RR}{\mathbb{R}}
\newcommand{\QQ}{\mathbb{Q}}
\newcommand{\ZZ}{\mathbb{Z}}
\newcommand{\Can}{2^{\omega}}

\newcommand{\pre}[2]{{}^{#1} #2}
\newcommand{\set}[2]{\{ #1 \mid #2 \}}

\newcommand{\PI}[2]{\ensuremath{\boldsymbol\Pi_{#2}^{#1}}}
\newcommand{\SI}[2]{\ensuremath{\boldsymbol\Sigma_{#2}^{#1}}}

\newcommand{\lh}{\operatorname{lh}}
\newcommand{\dom}{\operatorname{dom}}

\newcommand{\isom}{\cong}

\newtheorem{theorem}{Theorem}[section]
\newtheorem{lemma}[theorem]{Lemma}
\newtheorem{corollary}[theorem]{Corollary}
\newtheorem{proposition}[theorem]{Proposition}
\newtheorem{question}[theorem]{Question}
\newtheorem{quest}{Question}
\newtheorem{fact}[theorem]{Fact}
\newtheorem{claim}{Claim}[theorem]
\theoremstyle{definition}
\newtheorem{defin}[theorem]{Definition}

\theoremstyle{remark}
\newtheorem{remark}[theorem]{Remark}

\begin{document}

\title{Polish metric spaces with fixed distance set}
\date{\today}
\author[R.~Camerlo]{Riccardo Camerlo}
\address{Dipartimento di matematica, Universit\`a di Genova, Via Dodecaneso 35, 16146 Ge\-no\-va --- Italy}
\email{camerlo@dima.unige.it}
\author[A.~Marcone]{Alberto Marcone}
\address{Dipartimento di scienze matematiche, informatiche e fisiche, Universit\`a di Udine, Via delle Scienze 208, 33100 Udine --- Italy}
\email{alberto.marcone@uniud.it}
\author[L.~Motto Ros]{Luca Motto Ros}
\address{Dipartimento di matematica \guillemotleft{Giuseppe Peano}\guillemotright, Universit\`a di Torino, Via Carlo Alberto 10, 10123 Torino --- Italy}
\email{luca.mottoros@unito.it}
\thanks{Camerlo's research was partially carried out while visiting the \'Equipe de logique at the Universit\'e de Lausanne.
Marcone's research was supported by PRIN 2012 Grant \lq\lq Logica, Modelli e
Insiemi\rq\rq\ and departmental PRID funding \lq\lq HiWei --- The higher
levels of the Weihrauch hierarchy\rq\rq. Motto Ros was supported by the Young
Researchers Program ``Rita Levi Montalcini'' 2012 through the project ``New
advances in Descriptive Set Theory''.}

\begin{abstract}
We study Polish spaces for which a set of possible distances \( A \subseteq
\RR^+ \) is fixed in advance. We determine, depending on the properties of
\( A \), the complexity of the collection of all Polish metric spaces with
distances in $A$, obtaining also example of sets in some Wadge classes
where not many natural examples are known. Moreover we describe the
properties that \( A \) must have in order that all Polish spaces with
distances in that set belong to a given class, such as zero-dimensional,
locally compact, etc. These results lead us to give a fairly complete
description of the complexity, with respect to Borel reducibility and again
depending on the properties of \( A \), of the relations of isometry and
isometric embeddability between these Polish spaces.
\end{abstract}

\maketitle

\tableofcontents

\section{Introduction}

A common problem in mathematics is to classify interesting objects up to some
natural notion of equivalence. More precisely, one considers a class of
objects \( X \) and an equivalence relation \( E\) on \( X \), and tries to
find a set of complete invariants \( I \) for \( (X,E) \). To be of any use,
such an assignment of invariants should be as simple as possible. In most
cases, both \( X \) and \( I \) carry some intrinsic Borel structures, so
that it is natural to ask the assignment to be a Borel measurable map.

A classical example  is the problem  of classifying separable complete metric
spaces, called \emph{Polish metric spaces}, up to isometry. In
\cite{gromov1999} Gromov showed for instance that one can classify compact
Polish metric spaces using (essentially) elements of \( \RR \) as complete
invariants; in this case, one says that the corresponding
classification problem is smooth. However, as pointed out by Vershik in
\cite{vershik1998} the problem of classifying arbitrary Polish metric spaces
is \guillemotleft an enormous task\guillemotright, in particular it is far
from being smooth. Thus it is natural to ask ``how complicated'' is such a
classification problem.

A natural tool for studying the complexity of classification problems is the
notion of Borel reducibility introduced in \cite{Friedman1989}
and in \cite{HKL}: we say that a classification problem \( (X,E) \) is
\emph{Borel reducible} to another classification problem \( (Y,F) \) (in
symbols, \( E \leq_B F \)) if there exists a Borel measurable function \( f
\colon X \to Y \) such that \( x \mathrel{E} x' \iff f(x) \mathrel{F} f(x')
\) for all \( x,x' \in X \). Intuitively, this means that the classification
problem \( (X,E ) \) is not more complicated than \( (Y,F) \): in fact, any
assignment of complete invariants for \( (Y,F) \) may be turned into an
assignment for \( (X,E) \) by composing  with \( f \). A comprehensive
reference for the theory of Borel reducibility is \cite{gaobook}.

In the seminal \cite{Gao2003} (see also~\cite{Clemens2001,clemensisometry}),
Gao and Kechris were able to determine the exact complexity of the
classification problem for isometry on arbitrary Polish metric spaces with
respect to Borel reducibility: it is Borel bireducible with the most complex
orbit equivalence relation (so every equivalence relation induced by a Borel
action of a Polish group on a Polish space Borel reduces to it). However they
left the open problems of establishing the complexity of isometry on locally
compact ultrametric and zero-dimensional Polish spaces. We have been able to
solve the first of these questions in \cite{ultrametric} using an approach
that goes back to Clemens \cite{ClemensPreprint} and Gao and Shao
\cite{GaoShao2011}: Clemens studied the complexity of isometry on the
collection of Polish metric spaces using only distances in a set \( A
\subseteq \RR^+ \) fixed in advance, while Gao and Shao considered the
restriction of Clemens' problem to ultrametric Polish spaces.

We answered the questions left open by Gao and Shao in \cite{ultrametric},
where we focused on the study of ultrametric Polish spaces with a fixed set
of distances and, as a byproduct, we showed that isometry on locally compact
(and even discrete) ultrametric Polish spaces is Borel bireducible with
countable graph isomorphism. In this paper we instead settle various
problems, or provide new proofs for known results, about \emph{arbitrary}
Polish metric spaces with a fixed set of distances.

Let \( \RR^+ = \{ r \in \RR \mid r \geq 0 \} \). Let \( (X,d_X) \) be a
Polish metric space, i.e.\ a separable space with a complete metric \( d_X \)
(which often is left implicit). We denote by \( D(X) \) the set of distances
that are realized in \( X \), i.e.
\[
D(X) = \{ r \in \RR^+ \mid \exists x,y \in X (d_X(x,y) = r) \}.
\]
All metric spaces $X$ we consider are always assumed to be nonempty, so that
$0 \in D(X)$.

\begin{defin}
We say that \( A \subseteq \RR^+ \) is a \emph{distance set} if $A =D(X)$ for
some Polish metric space $X$. When $A =D(X)$ we say that \emph{$A$ is
realized by $X$}. Let \( \D \) denote the set of all distance sets $A
\subseteq \RR^+$.
\end{defin}

Clemens characterized the members of \( \D \) in his PhD thesis.

\begin{theorem}[{\cite[Theorem 4.3]{ClemensThesis}}] \label{clemensrealized}
Let $A \subseteq \RR^+$. Then $A$ is a distance set if and only if $A$ is
analytic, $0 \in A$, and either $A$ is countable or $0$ is a limit point of
$A$.
\end{theorem}

Clemens studied also for which $A \in \D$, given a property of Polish spaces
(like being locally compact, or $\sigma$-compact, or discrete, or countable,
and so on) some Polish metric space with this property has distance set $A$
(see Theorem \ref{necessary} below). Here we consider the following dual
question:

\begin{quest}\label{quest1}
For which $A \in \D$ \emph{every} Polish metric space with distance set $A$
has a given property?
\end{quest}

We answer this question in Section \ref{realizable}, and in particular in
Theorem~\ref{proponlyultrametric}. Our results show in particular that lower
bounds for the complexity of the restriction of isometry to zero-dimensional
Polish metric spaces (one of the problems of Gao and Kechris) can be obtained
by classifying the restriction of isometry to spaces with a fixed distance
set which is dense in some right neighborhood of $0$ but does not contain any
such neighborhood.

Another natural question is the following:

\begin{quest}\label{quest2}
Given $A \in \D$, what is the complexity of the collection $\V^\star_A$ of
Polish metric spaces having distance set $A$? What about the collection
$\V_A$ of Polish metric spaces having distance set included in $A$ (in which
case we can drop the requirement $A \in \D$)?
\end{quest}

This (and Question \ref{quest3} below) requires to view Polish metric spaces
as members of some hyperspace of all Polish metric spaces: we describe the
set-up in Section \ref{term} and answer quite satisfactorily Question
\ref{quest2} in Section \ref{complexity}. In particular
Theorems~\ref{vaborel}(2) and~\ref{vastarborel}(1) characterize when $\V_A$
and $\V^\star_A$ are standard Borel. Tables \ref{tableM} and \ref{tableM*}
summarize our results for the complexity of $\V_A$ and $\V^\star_A$ when $A
\in \D$.

As a corollary, in Theorem \ref{Urysohn} we also extend the characterization
of the distance sets $A$ which admit an $A$-Urysohn space obtained by Sauer
\cite{Sauer}\footnote{We thank Joseph Zielinski for directing us to Sauer's
paper.}.

The last main questions we deal with are the original motivation for this
research:

\begin{quest}\label{quest3}
Given $A \in \D$, what is the complexity with respect to Borel reducibility
of isometry and isometric embeddability restricted to $\V^\star_A$ (denoted
respectively ${\isom_A^\star}$ and $\sqsubseteq_A^\star$)?
What about the same problem for isometry and isometric embeddability
restricted to $\V_A$ (denoted respectively ${\isom_A}$ and $\sqsubseteq_A$)?
\end{quest}

We study this question in Section \ref{polish}, and our main results include
the following:
\begin{itemize}
\item The fact that ${\sqsubseteq_A}$ and ${\sqsubseteq^\star_A}$ have the
    same complexity for all $A \in \D$ (Corollary \ref{corsqsubseteqA}) and
    that ${\isom_A}$ and ${\isom_A^\star}$ have the same complexity when
    $A$ is countable (Theorem \ref{thm:countable}).
\item The classification with respect to Borel reducibility of \(
    \isom^\star_A \), depending on the properties of \( A \)
    (Theorem~\ref{isomstarA}). In particular, we characterize when
    countable graph isomorphism Borel reduces to \( \isom^\star_A \)
    (Theorem~\ref{appendixisom}).
\item The exhaustive description of the complexity with respect to Borel
    reducibility of \( \sqsubseteq^\star_A \)
    (Theorem~\ref{sqsubseteqstarA}).
\item The fact that whenever \( \sqsubseteq^\star_A \) is complete
    analytic, it has also the stronger property of being invariantly
    universal (Theorem~\ref{thm:invuniv}).
\end{itemize}

The first item substantially enriches the picture obtained by Clemens in
\cite{ClemensPreprint}, almost completely solving his original problem about
isometry. We also answer some of the other questions contained in
\cite{ClemensPreprint} (Proposition~\ref{questionclemens},
Theorem~\ref{thm:isomA} and Theorem~\ref{thm:countable}), and their analogues
concerning  isometric embeddability (Proposition~\ref{questionclemens} and
Corollary~\ref{corsqsubseteqA}).

\section{Preliminaries}\label{term}
If $\mathcal{A}$ is a countably generated $\sigma$-algebra of subsets of $X$
that separates points we refer to the members of $\mathcal{A}$ as Borel sets
(indeed, as shown e.g.\ in \cite[Proposition 12.1]{Kechris1995}, in this case
$\mathcal{A}$ is the collection of Borel sets of some separable metrizable
topology on $X$), and to \( (X, \mathcal{A}) \) as a Borel space. The Borel
space $(X, \mathcal{A})$ is standard if $\mathcal{A}$ is the collection of
Borel sets of some Polish (i.e.\ separable and completely metrizable)
topology on $X$. A map between two Borel spaces is Borel if the preimages of
Borel sets of the target space are Borel sets of the domain.

We denote by $\SI11(X)$ the family of subsets of the standard Borel spaces
$X$ which are Borel images of a standard Borel space. For $n>0$, $\PI1n(X)$
is the class of all complements of sets in $\SI1n(X)$, and $\SI1{n+1}(X)$ is
the family of Borel images of a set in $\PI1n(Y)$ for some standard Borel
space $Y$. We have $\SI1n \cup \PI1n \subseteq \SI1m \cap \PI1m$ whenever
$n<m$, and for uncountable standard Borel spaces the inclusion is strict.
This hierarchy is the \emph{projective hierarchy}. \SI11 and \PI11 sets are
called resp.\ \emph{analytic} and \emph{coanalytic} sets. The class of
differences of two analytic subsets (equivalently: of intersections of an
analytic and a coanalytic subset) of $X$ is denoted $D_2( \SI{1}{1} )(X)$.

We extend these notions also to Borel spaces $X$ which are not necessarily
standard. In particular we say that $A \subseteq X$ is analytic (or \SI11) if
there exists a standard Borel space $Y \supseteq X$ such that the Borel
subsets of $X$ are the intersections of the Borel subsets of $Y$ with $X$ and
$A$ is the intersection of some $B \in \SI11(X)$ with $X$.\smallskip

If $\mathbf{\Gamma}$ is a class of sets in Borel spaces closed under Borel
preimages (like \SI1n and \PI1n), $Y$ is a standard Borel space and $A
\subseteq Y$, we say that $A$ is \emph{Borel $\mathbf{\Gamma}$-hard} if for
every $B \in \mathbf{\Gamma} (X)$, where $X$ is a standard Borel space, is
\emph{Borel Wadge reducible} to $A$, i.e.\ there exists a Borel function $f:
X \to Y$ such that $f^{-1}(A) = B$. We say that $A$ is \emph{Borel
$\mathbf{\Gamma}$-complete} if, in addition, $A \in \mathbf{\Gamma} (Y)$. If
$A$ is Borel $\mathbf{\Gamma}$-hard and $A$ is Borel Wadge reducible to $B$,
then $B$ is Borel $\mathbf{\Gamma}$-hard as well: this is the typical way to
prove hardness results.

The classes $\mathbf{\Gamma}$ we are interested in are closed under Borel
preimages and such that either $\mathbf{\Gamma}$ or its dual
$\check{\mathbf{\Gamma}}$ (consisting of the complements of the elements of
$\mathbf{\Gamma}$) is closed under intersection with $\PI11$ sets. For these
classes and any Polish topology compatible with the standard Borel spaces,
Borel $ \mathbf{\Gamma } $-hardness can be witnessed by continuous functions:
see \cite{kechri}, where this fact is stated for the class $ \PI11$, but the
argument actually works under our more general assumptions on
$\mathbf{\Gamma}$. Therefore Borel $\mathbf{\Gamma}$-hardness and Borel
$\mathbf{\Gamma}$-completeness coincide with $\mathbf{\Gamma }$-hardness and
$\mathbf{\Gamma }$-completeness, which are the notions used when dealing with
Wadge reducibility. For this reason we drop Borel from this terminology.

Most results in Section \ref{complexity} state that a collection of Polish
metric spaces is $\mathbf{\Gamma}$-complete for some $\mathbf{\Gamma}$, and
thus pinpoint the complexity of that particular collection by showing that it
belongs to $\mathbf{\Gamma}$ and not to any simpler class. When $
\mathbf{\Gamma } \neq \check{ \mathbf{\Gamma }}$, this implies in particular
that such a collection does not belong to $ \check{\mathbf{\Gamma}
}$.\smallskip

Borel Wadge reducibility can be generalized from sets to binary (and in fact,
$n$-ary for any $n$) relations as follows. Let $R$ and $S$ be binary
relations on Borel spaces $X$ and $Y$, respectively. We say that $R$ is
\emph{Borel reducible} to $S$, and we write $R \leq_B S$, if there is a Borel
function $f: X\to Y$ such that $x \mathrel{R} x'$ if and only if $f(x)
\mathrel{S} f(x')$ for all $x,x'\in X$. If $R\leq_BS$ and $S \leq_B R$ we say
that $R$ and $S$ are \emph{Borel bireducible} and we write $R \sim_B S$. If
on the other hand we have $R\leq_BS$ and $S\nleq_BR$ we write $R <_B S$.

If $\mathbf{\Gamma}$ is a class of binary relations on standard Borel spaces
and $S \in \mathbf{\Gamma}$, we say that $S$ is \emph{complete for
$\mathbf{\Gamma}$} if $R \leq_B S$ for all $R \in \mathbf{\Gamma}$. Some
relevant classes $\mathbf{\Gamma}$ one might consider are the collection of
all analytic equivalence relations and the collection of all analytic
quasi-orders. (Recall that a quasi-order is a reflexive and transitive binary
relation.) An example of a complete quasi-order for the latter class is
embeddability between countable graphs, see \cite{louros}. When
$\mathbf{\Gamma}$ is the class of orbit equivalence relations (that is, those
Borel reducible to a relation induced by a Borel action of a Polish group on
a standard Borel space) a complete element for $\mathbf{\Gamma}$ is isometry
on arbitrary Polish spaces, see \cite{Gao2003}. Another important example is
the class of equivalence relations classifiable by countable structures, that
is those Borel reducible to isomorphism on countable structures. The
canonical example of an equivalence relation complete for this class is
countable graph isomorphism, see \cite{Friedman1989}.

We reserve the term \lq\lq complete for $\mathbf{\Gamma}$\rq\rq\ for
relations defined on some standard Borel space. In Section \ref{polish}, we
often consider analytic relations on Borel spaces which (by the results of
Section \ref{complexity}) are not standard. In these cases we rather state
that a relation is \emph{Borel bireducible with a complete for
$\mathbf{\Gamma}$ relation}.\medskip

We denote by $\isom$ and $\sqsubseteq$ the relations of isometry and
isometric embeddability between metric spaces. Recall that a metric space is
Polish if and only if it is isometric to an element of \( F(\mathbb{U}) \),
the collection of all nonempty closed subsets of the Urysohn space \(
\mathbb{U} \) (here we differ slightly from \cite{Kechris1995}, where \(
F(\mathbb{U}) \) includes the empty set). The space $F(\mathbb{U})$ is
endowed with the Effros Borel structure, which turns it into a standard Borel
space: the hyperspace containing all Polish metric spaces up to isometry.
Notice that $\isom$ and $\sqsubseteq$ are analytic relations on
$F(\mathbb{U})$. We fix also a sequence of Borel functions \( (\psi_n)_{n \in
\omega} \) from \( F(\mathbb{U}) \) into \( \mathbb{U} \) such that \( \{
\psi_n(X) \mid n \in \omega \} \) is dense in \( X \) for every \( X \in
F(\mathbb{U}) \), see \cite[Theorem 12.13]{Kechris1995}.

\begin{remark}\label{rem:coding}
Another possible coding for Polish metric spaces (used e.g.\ in
\cite{ClemensPreprint}) is sometimes convenient. In this approach a Polish
metric space $U$ is coded by an element $M$ of a suitable $\mathcal M$, which
is a closed subset of $\pre{\omega\times\omega }{ \RR } $: $U$ is the
completion of a set of points $\{ x_i \mid i\in\omega \}$ such that the
distance between $x_i$ and $x_j$ equals $M(i,j)$. As explained in
\cite[Section 2]{lmrnew}, this coding is equivalent to the one introduced
above, in the sense that there are Borel functions $\Phi: F(\mathbb{U}) \to
\mathcal M$ and $\Psi: \mathcal M \to F(\mathbb{U})$ such that $\Phi(X)$
codes a space isometric to $X$ and $\Psi(M)$ is isometric to the space coded
by $M$.
Therefore the results can be transferred between the two settings. In
particular, it is often easier to check that certain maps are
Borel-measurable using $\mathcal M$ rather than $F(\mathbb{U})$ (see e.g.\
the proof of Proposition~\ref{1531426}).
\end{remark}

Using the coding in $F( \mathbb U )$ we have the following formalizations of the
collections of Polish metric spaces using a prescribed set of distances.

\begin{defin}
Given $A \subseteq \RR^+$, let
\[
\V_A= \set{X\in F( \mathbb U )}{D(X)\subseteq A} \text{ and }
\V_A^\star = \set{X\in F( \mathbb U )}{D(X)=A} .
\]
Let the equivalence relations $ \isom_A $ and $ \isom_A^\star$ and the
quasi-orders $\sqsubseteq_A$ and $\sqsubseteq_A^\star$ be the restrictions of
isometry and isometric embeddability to $\V_A$ and $\V_A^\star$.
\end{defin}

The relations $ \isom_A $, $ \isom_A^\star$, $\sqsubseteq_A$, and
$\sqsubseteq_A^\star$ are defined on Borel spaces which are not necessarily
standard (not even when $A$ is countable, as we will show). We will discuss
the complexity of $\V_A$ and $\V_A^\star$ at length in Section
\ref{complexity}.

In \cite{ultrametric} we studied the restrictions of isometry and isometric
embeddability to
\[
\U_A = \set{X \in \V_A}{d_X \text{ is an ultrametric}} \text{ and }
\U^\star_A = \set{X \in \V^\star_A}{d_X \text{ is an ultrametric}}.
\]
By Theorem \ref{necessary} below the latter is nonempty exactly when $A \in
\D$ is countable. In contrast with the results of Section \ref{complexity},
$\U_A$ and $\U_A^{\star}$ are both standard Borel spaces (see
\cite[Proposition 4.5]{ultrametric}). We first considered the case of $A$
ill-founded (with respect to the standard ordering of the reals) and,
extending results of \cite{Gao2003}, proved the following.

\begin{theorem}[{\cite[Corollary 5.7, and Theorems 6.3 and 6.4]{ultrametric}}]\label{isomuniversalstar}
Let $A \in \D$ be countable and ill-founded. Then:
\begin{enumerate}
\item Isometry on $\U_A$ and on $\U_A^{\star}$ are both complete for
    equivalence relations classifiable by countable structures.
\item Isometric embeddability on $\U_A$ and on $\U_A^{\star}$ are both
    complete for analytic quasi-orders.
\end{enumerate}
\end{theorem}

Then we dealt with well-founded $A$'s. Lemma 4.11 of \cite{ultrametric}
implies that the complexities of isometry and isometric embeddability on
$\U_A$ and $\U_A^\star$ for $A \in \D$ well-founded depend only on the order
type of $A$. If $\alpha$ with $1 \leq \alpha<\omega_1$ is the order type of
$A$ we can then write, as in \cite{ultrametric}, ${\isom_{\alpha}}$,
${\isom^\star_{\alpha}}$, ${\sqsubseteq_\alpha}$, and
${\sqsubseteq^\star_\alpha}$ in place of ${{\isom} \restriction \U_A}$,
${{\isom} \restriction \U^\star_A}$, ${{\sqsubseteq} \restriction \U_A}$ and
${{\sqsubseteq} \restriction \U^\star_A}$, respectively. Our results include
the following.

\begin{theorem}[{\cite[Lemma 5.8, and Theorems 5.15 and 6.11]{ultrametric}}]\label{lemmaequivalence}
\emph{ }

\begin{enumerate}[(1)]
\item For every $\alpha$ such that \( 1\leq\alpha < \omega_1 \) we have ${
    \isom_{\alpha }} \sim_B {\isom_{\alpha }^\star}$ and these equivalence
    relations are classifiable by countable structures.
\item The relations \( \isom_\alpha \), for $1\leq\alpha <\omega_1$, form a
    strictly increasing chain under $\leq_B$ of Borel equivalence relations
    which is cofinal below countable graph isomorphism (i.e.\ cofinal among
    Borel equivalence relations classifiable by countable structures).
\item For every $\alpha$ such that \( 1\leq\alpha < \omega_1 \) we have \(
    {\sqsubseteq_\alpha} \sim_B {\sqsubseteq^\star_\alpha} \).
\item Let $1\leq \alpha < \omega_1$. Then
\begin{enumerate}[(i)]
\item if \( \alpha \leq \omega \), then \( \sqsubseteq_\alpha \) is Borel;
\item if $\alpha > \omega$, \( \sqsubseteq_\alpha \) contains both upper
    and lower cones that are \SI11-complete, and hence
    $\sqsubseteq_\alpha$ is analytic non-Borel;
\item all classes of the equivalence relation induced by \(
    \sqsubseteq_{\omega+1} \) are Borel, hence \( \sqsubseteq_{\omega+1}
    \) is not complete for analytic quasi-orders;
\item for all \( \alpha < \beta \leq \omega+2 \), \( {\sqsubseteq_\alpha}
    <_B {\sqsubseteq_\beta} \).
\end{enumerate}
\end{enumerate}
\end{theorem}

The problem of establishing the exact complexity of $\sqsubseteq_\alpha$ for
$\alpha \geq \omega+2$ is still open.

Basic tools to change from a set of distances to another one are \emph{metric
preserving functions}, i.e.\ functions $f \colon A \to \RR^+$ such that for
every metric $d$ on a space $X$ with range contained in $A$ we have that $f
\circ d$ is still a metric on $X$. There is a vast literature about metric
preserving functions defined on the whole $\RR^+$ (see \cite{Dobos,Cor} for
surveys). Since we are dealing with Polish metric spaces we introduce the
following definition, where we consider functions with a possibly proper
subset of $\RR^+$ as domain.

\begin{defin}
A function $f \colon A \to \RR^+$ is \emph{Polish metric preserving} if for
every complete metric $d$ with range contained in $A$ on a Polish space $X$
we have that $(X, f \circ d)$ is still a Polish metric space.
\end{defin}

\begin{proposition}\label{Polish_metric_pres}
A function $f:A\to \RR^+$ is Polish metric preserving if and only if it is
metric preserving and for every sequence $(x_n)_{n \in \omega}$ in $A$,
\begin{equation} \label{15031241}
\lim_{n\to\infty }x_n=0\quad \text{if and only if} \quad\lim_{n\to\infty
}f(x_n)=0.
\end{equation}
\end{proposition}
\begin{proof}
Let $f$ be Polish metric preserving. First we prove that $f$ is metric
preserving. Let $(X,d)$ be a metric space with distances in $A$: if
$(X,f\circ d)$ were not a metric space this would be witnessed on a subset
$X' \subseteq X$ of size two or three; then $(X',d)$ is Polish but
$(X',f\circ d)$ is not even a metric space.

Fix now a sequence $(x_n)_{n \in \omega}$ in $A$. If $(x_n)_{n \in \omega}$
converges to $0$ but $(f(x_n))_{n \in \omega}$ does not, then it can be
assumed that $(x_n)_{n \in \omega}$ is strictly decreasing and $(f(x_n))_{n
\in \omega}$ is bounded away from $0$. Let $d$ the metric on ${}^{\omega
}\omega $ defined by letting, for distinct $\alpha ,\beta\in {}^{\omega
}\omega $, $d(\alpha ,\beta )=x_n$ where $n$ is least such that $\alpha
(n)\ne\beta (n)$. Then $({}^{\omega }\omega ,d)$ is a Polish metric space,
while in $({}^{\omega }\omega ,f\circ d)$ every point is isolated, so
$({}^{\omega }\omega ,f\circ d)$ is not separable and hence not Polish.
Conversely, if $(x_n)_{n \in \omega}$ does not converge to $0$ but
$(f(x_n))_{n \in \omega}$ does, it can be assumed that $(x_n)_{n \in \omega}$
is bounded away from $0$. Let $X=\{ x_n\mid n\in\omega\} $, and define the
distance $d$ on $X$ by letting $d(x_n,x_m)=\max\{x_n,x_m\}$ if $x_n \neq
x_m$. Then $(X,d)$ is a discrete Polish ultrametric space. Then the sequence
$(x_n)_{n \in \omega}$ is a Cauchy sequence in $(X,f\circ d)$ but it does not
converge, as $(X,f\circ d)$ is discrete.

Assume now that $f$ is metric preserving and that condition \eqref{15031241}
holds for every sequence $(x_n)_{n \in \omega}$ in $A$. This means that if
$(X,d)$ is a metric space with distances in $A$, then the identity is a
homeomorphism between $(X,d)$ and $(X,f\circ d)$. In particular, if $(X,d)$
is Polish, then $(X,f\circ d)$ is separable; to conclude that $(X,f\circ d)$
is complete too, notice that a sequence in $X$ is $d$-Cauchy if and only if
it is $f\circ d$-Cauchy.
\end{proof}

A Polish metric preserving function $f: A \to A'$ transforms a space $(X,d)
\in \V_A$ into the space $(X, f \circ d) \in \V_{A'}$, which is homeomorphic
to $(X,d)$ via the identity function by the previous characterization. The
following proposition ensures that this transformation is always Borel.

\begin{proposition}\label{prop:automaticBorel}
Let $A \in \D$. Every Polish metric preserving $f:A \to \RR^+$ is Borel-measurable.
\end{proposition}
\begin{proof}
Fix $(X,d) \in \V_A^\star$ and a countable dense $D = \{x_i \mid i \in
\omega\} \subseteq X$. Let $r_{i,j} = d(x_i, x_j)$ and $r'_{i,j} =
f(r_{i,j})$. Define $F \subseteq \RR^+ \times \RR^+$ by setting $(a,b) \in F$
if and only if
\begin{align*}
\exists (i_k)_{k \in \omega}, (j_k)_{k \in \omega} \big[ & (x_{i_k})_{k \in \omega} \text{ and } (x_{j_k})_{k \in \omega}
                         \text{ are Cauchy sequences in } X\\
& \qquad \land \lim r_{i_k,j_k} = a \land \lim r'_{i_k,j_k} = b \big].
\end{align*}
The set $F$ is clearly analytic. Since $\mathrm{id}_X$ is a homeomorphism
between $(X,d)$ and $(X, f \circ d)$, the set $F$ is the graph of $f$. Using
the fact that Souslin's Theorem holds for analytic spaces (\cite[Exercise
28.3]{Kechris1995}), we have that the proof of \cite[Theorem
14.12]{Kechris1995} shows that functions on analytic spaces with analytic
graphs are Borel. Therefore, since $A$ is analytic by Theorem
\ref{clemensrealized}, we have that $f$ is Borel.
\end{proof}

\section{Distance sets of particular Polish metric spaces}\label{realizable}
Beside proving Theorem \ref{clemensrealized} characterizing distance sets,
Clemens also characterized the \( A \in \D \) which can be realized by Polish
metric spaces in a given class.

\begin{theorem}[\cite{ClemensThesis}]\label{necessary}
Let $A \in \D$. Then:
\begin{enumerate}
\item \( \V^\star_A \) always contains a zero-dimensional Polish metric
    space;
\item \( \V^\star_A \) contains a Polish ultrametric space if and only if
    it contains a discrete Polish metric space, if and only if \( A \) is
    countable;
\item \( \V^\star_A \) contains a connected Polish metric space if and only
    if it contains a path-connected Polish metric space if and only if \( A
    \) is an interval with left endpoint \( 0 \);
\item \( \V^\star_A \) contains a compact Polish metric space if and only
    if \( A \) is compact and either it is finite or it has \( 0 \) as
    limit point;
\item \( \V^\star_A \) contains a locally compact Polish metric space if
    and only if it contains a \( \sigma \)-compact Polish metric space, if
    and only if \( A \) is either countable or it is $\sigma$-compact and
    has \( 0 \) as limit point.
\end{enumerate}
\end{theorem}

We now consider the dual problem of determining when \( A \in \D\) is
realized \emph{only} by Polish metric spaces in a given class. We need the
following construction, which will be used repeatedly throughout the paper.

Let \( (X, d_X ) \) and \( (Y, d_Y ) \) be metric spaces. Given two points \(
\bar{x} \in X \) and \( \bar{y} \in Y \) and a real \( r>0 \) we can extend
the metrics \( d_X \) and \( d_Y \) to the disjoint union \( Z = X \cup Y \)
by setting \( d_Z(x,y) = \max\{ d_X(x, \bar{x}), d_Y(y, \bar{y}), r\} \) for
\( x \in X \) and \( y \in Y \).

\begin{lemma}\label{oplus}
The function \(d_Z \) defined above is a metric. Moreover $(Z,d_Z)$ is Polish
whenever \( (X, d_X ) \) and \( (Y, d_Y ) \) are Polish.
\end{lemma}
\begin{proof}
To prove that $d_Z$ is a metric, we just need to check $d_Z(a,b) \leq
d_Z(a,c) + d_Z(c,b)$ for distinct points $a,b,c \in X \cup Y$. If $a,b,c \in
X$ or $a,b,c \in Y$ this is trivial.

Now assume $a,b \in X$ and $c \in Y$: then $d_Z(a,b) = d_X(a,b) \leq
d_X(a,\bar{x}) + d_X(\bar{x},b) \leq d_Z(a,c) + d_Z(c,b)$. The case $a,b \in
Y$ and $c \in X$ is symmetric.

Assume now $a \in X$ and $b \in Y$ (the symmetric case is analogous). We
distinguish three cases.
\begin{itemize}
  \item Assume $d_Z(a,b)=r$. If $c \in X$ then $d_Z(c,b) \geq r$, while if
      $c \in Y$ then $d_Z(a,c) \geq r$. In both cases $d_Z(a,b)=r \leq
      d_Z(a,c) + d_Z(c,b)$.
  \item Assume $d_Z(a,b)=d_X(a,\bar{x})$. If $c \in X$ then $d_Z(a,b) =
      d_X(a,\bar{x}) \leq d_X(a,c) + d_X(c,\bar{x}) \leq d_Z(a,c) +
      d_Z(c,b)$. If instead $c \in Y$ then $d_X(a,\bar{x}) \leq d_Z(a,c)$,
      whence $d_Z(a,b) = d_X(a,\bar{x}) \leq d_Z(a,c) + d_Z(c,b)$.
  \item The case $d_Z(a,b)=d_Y(b,\bar{y})$ is similar to the previous one.
\end{itemize}

Polishness is preserved because every Cauchy sequence in $Z$ is eventually
contained either in $X$ or in $Y$, so the construction does not add new
limits of Cauchy sequences.
\end{proof}

We denote the metric space $(Z,d_Z)$ by $X \oplus_r Y$, omitting reference to
$\bar{x}$ and $\bar{y}$ because the choice of these two points will be
irrelevant in most of our applications.
The following proposition shows that given \( r > 0 \), the map \(
F(\mathbb{U}) \times F(\mathbb{U}) \to F(\mathbb{U}) \) sending \( X,Y \in
F(\mathbb{U}) \) to \( X \oplus_r Y \) may be construed as a Borel function.

\begin{proposition}\label{1531426}
There is a Borel-measurable function $f:F( \mathbb U )\times F( \mathbb U
)\to F( \mathbb U )$ such that $f(X,Y)$ is isometric to $X\oplus_rY$ for some
choice of the gluing points $ \bar x \in X$ and $\bar y \in Y$.
\end{proposition}
\begin{proof}
This is immediate using Remark \ref{rem:coding}, as it is easy to define a
Borel counterpart of this function from $\mathcal M \times \mathcal M$ to
$\mathcal M$.
\end{proof}

The most important property of this construction is the following:

\begin{fact}\label{fact:oplus}
For every $r$ we have $D(X \oplus_r Y) = D(X) \cup D(Y) \cup \{r\}$.
\end{fact}

\begin{defin}
\( A \in \D \) is \emph{well-spaced} if \( r < r' \) implies \( 2r < r' \)
for all \( r,r' \in A \).
\end{defin}

Notice that if \( A \) is well-spaced and infinite then \( A \setminus \{0\}
\) is either a decreasing sequence converging to $0$, an unbounded increasing sequence, or the union of these two.
This  follows from the fact that if $A$ is well-spaced  then for $n\in \ZZ $ the set $A \cap [2^n,2^{n+1})$ contains at most one point.

\begin{theorem} \label{proponlyultrametric}
Let \( A \in \D \).
\begin{enumerate}
\item All spaces in $\V^\star_A$ are zero-dimensional if and only if \( A
    \) does not contain a right neighborhood of \( 0 \).
\item All spaces in $\V^\star_A$ are ultrametric if and only if \( A \) is
    well-spaced.
\item All spaces in $\V^\star_A$ are discrete if and only if they are all
    locally compact, if and only if they are all \( \sigma \)-compact, if
    and only if $0$ is isolated in \( A \).
\item All spaces in $\V^\star_A$ are connected if and only if they are all
    compact, if and only if they are all singletons, if and only if $A= \{0\}$.
\end{enumerate}
All the above characterizations remain true if we replace $\V^\star_A$ with
$\V_A$.
\end{theorem}
\begin{proof}
To prove all forward directions we construct spaces of the form $X \oplus_r
Y$ with $r \in A$, $X \in \V^\star_A$, and $Y \in \V_A$ lacking the relevant
topological properties. We always have $X \oplus_r Y \in \V^\star_A$ by Fact
\ref{fact:oplus}.

(1) Suppose that \( A \) contains an interval \( [0,r] \) for some \( r > 0
\). Fix $X \in \V^\star_A$ and let \( Y = [0,r] \subseteq \RR\) with the
usual metric. Then \( X \oplus_r Y\) belongs to $\V^\star_A$ but is not
zero-dimensional.

Conversely, if \( 0 \) is a limit point of \( \RR^+ \setminus A \), then for
any $X \in \V^\star_A$ the collection of balls with radius in \( \RR^+
\setminus A \) is a clopen basis for $X$.

(2) Recall that a space is ultrametric if and only if every triangle is
isosceles with legs not shorter than the base. Suppose that \( A \) is not
well-spaced and pick \( r,r' \in A \) with \( r < r' \leq 2r \). Fix $X \in
\V^\star_A$ and let \( Y \) be a triangle with two sides of length \( r \)
and one of length \( r' \). Then \( X \oplus_r Y \) belongs to $\V^\star_A$
but is not ultrametric.

Conversely, if \( X  \in \V^\star_A\) is not ultrametric, then  it must
contain a triangle with sides of length \( r'' \leq r < r' \). Then by the
triangle inequality \( r' \leq r'' + r \leq 2 r \), and hence \( A \) is not
well-spaced.

(3) Suppose that  \( A \) contains  a decreasing sequence \( (r_n)_{n \in
\omega} \) converging to \( 0 \). Fix again \( X  \in \V^\star_A\) and let \(
Y \) be the Baire space \( \pre{\omega}{\omega} \) equipped with the metric
defined by $d_Y(y,y') = r_n$ where $n$ is least such that $y(n) \neq y'(n)$.
Then the space \( X \oplus_{r_0} Y \) belongs to $\V^\star_A$ but is not \(
\sigma \)-compact (here we are also using the fact that $Y$ is closed in \( X
\oplus_{r_0} Y \)), and hence, by separability, neither locally compact nor discrete.

Conversely, if $0$ is isolated in \( A \), then any \( X  \in \V^\star_A\) is
discrete, and thus also locally compact and \( \sigma \)-compact.

(4) Suppose that $r>0$ belongs to $A$ and $X \in \V^\star_A$. Let $Y$ be the
countable space with all distinct points at distance $r$: then $X \oplus_r Y$
belongs to $\V^\star_A$ but is neither connected nor compact. The other
implications are obvious.\smallskip

Finally, consider the statements for $\V_A$. The forward directions follow
from $\V^\star_A \subseteq \V_A$. For the backward directions, let $X \in
\V_A$ and set $A' = D(X)$ so that $X \in \V^\star_{A'}$: since $A' \subseteq
A$ and all the stated conditions on \( A \) are inherited by subsets, the
results follow from what we proved above.
\end{proof}

Theorem~\ref{proponlyultrametric}(1) shows that restricting the attention to
Polish metric spaces with a fixed set of distance \( A \in \D \) may provide
useful information on the complexity of the isometry relation on
zero-dimensional Polish metric spaces: indeed, if \( A \) does not contain a
right neighborhood of \( 0 \), then \( \isom^\star_A \) is a lower bound for
such a relation. In contrast, (3) and (4) of Theorem
\ref{proponlyultrametric} imply that the approach of restricting isometry to
Polish metric spaces using a specific distance set cannot provide interesting
lower bounds for the complexity of locally compact or connected Polish metric
spaces.

\section{The complexity of $\V_A$ and $\V^\star_A$}\label{complexity}

We consider the problem of determining the complexity of the subspaces $\V_A$
and $\V^\star_A$ of $F(\mathbb{U})$, in particular characterizing when they
are standard Borel spaces. While it is worth studying $\V_A$ for any $A
\subseteq \RR^+$ with $0 \in A$, we have $\V^\star_A \neq \emptyset$ only
when $A \in \D$.

Notice the following fact, immediate from the definitions:

\begin{fact}\label{upperV}
If $A$ is analytic (in particular if $A \in \D$) then both $\V_A$ and
$\V^\star_A$ are \PI12, while when $A$ is Borel then $\V_A$ is \PI11. If $A$
is countable then $\V_A$ is \PI11 and $\V_A^\star$ belongs to $D_2( \SI11)$.
\end{fact}

The following reductions are easy to prove and very general.

\begin{proposition}\label{basicVA}
Let $0 \in A\subseteq \RR^+$.
\begin{enumerate}
  \item $A$ Borel reduces to $\V_A$;
  \item if $A \in \D$ then $\V_A$ Borel reduces to $\V^\star_A$.
\end{enumerate}
\end{proposition}
\begin{proof}
(1) One can define a Borel function $f: \RR^+ \to F( \mathbb U )$ such that
$f(0)$ is a singleton and otherwise $f(r)$ is a space consisting of two
points at distance $r$: then $A =f^{-1}( \V_A)$.

(2) If $A=\{ 0\} $, then $ \V_A= \V_A^{\star }$.
Otherwise, fix $Y \in \V^\star_A$ and $r \in A \setminus \{0\}$. The Borel map $X
\mapsto X \oplus_r Y$ reduces $\V_A$ to $\V^\star_A$.
\end{proof}

To obtain sharper results we make extensive use of the following Borel
construction of Polish metric spaces.

\begin{defin}\label{def:ts}
A triple $((r_n)_{n\in\omega }, (r'_n)_{n\in\omega }, x)$ is
\emph{tree-suitable} if
\begin{itemize}
\item $x>0$;
\item $(r_n)_{n\in\omega }$ is strictly decreasing and converges to $0$;
\item $(r'_n)_{n\in\omega }$ is strictly monotone and converges to $x$;
\item $r_0 < \min (x,r'_0)$;
\item $\forall n\in\omega\ |r'_n - x| <r_n$, so that $\forall n,m\
    |r'_n - r'_m| \leq \max \{ r_n,r_m \}$.
\end{itemize}
\end{defin}

In this case one can define an assignment $\Phi_{r_nr'_n}$ that sends a tree
$T \subseteq \omega^{<\omega}$ to some $\Phi_{r_nr'_n}(T) \in F( \mathbb U )$
which is isometric to the completion of $T \cup \{\ast\}$ under the metric
$d$ defined by setting $d(s,t) =r_n$ if $s,t \in T$ are distinct and $n$ is
largest such that $s \restriction n = t \restriction n$, and $d(s,\ast) =r'_{
\lh (s)}$ for $s \in T$. The last of the conditions in the definition of
tree-suitability ensures that $d$ satisfies the triangle inequality. Using
Remark \ref{rem:coding} and going through $\mathcal M$, we can assume that
$\Phi_{r_nr'_n}$ is Borel.

The main property of
$\Phi_{r_nr'_n}(T)$ is the following:

\begin{fact}\label{fact:ts}
If $((r_n)_{n\in\omega }, (r'_n)_{n\in\omega }, x)$ is tree-suitable then for
any tree $T \subseteq \omega^{<\omega}$:
\begin{itemize}
  \item $D(\Phi_{r_nr'_n}(T)) \subseteq \{ 0 \} \cup \set{r_n, r'_n}{n \in \omega} \cup
      \{x\}$;
  \item $x\in D(\Phi_{r_nr'_n}(T))$ if and only if $T$ is ill-founded.
\end{itemize}
\end{fact}

We first study the complexity of $\V_A$.

\begin{theorem}\label{vaborel}
Let $0 \in A\subseteq \RR^+$.
\begin{enumerate}
  \item If $A$ is not closed and $0$ is a limit point of $A$, then $\V_A$
      is \PI11-hard;
  \item $\V_A$ is Borel if and only if either $A$ is closed, or $A$ is
      Borel and $0$ is not a limit point of $A$;
  \item if $A$ is Borel but not closed and $0$ is a limit point of $A$ then
      $\V_A$ is \PI11-complete;
  \item if $A$ is \SI11-complete and $0$ is a limit point of $A$ then \(
      \V_A \) is \PI12-complete.
\end{enumerate}
\end{theorem}
\begin{proof}
(1) Fix $x \in \bar A \setminus A$, and sequences $(r_n)_{n\in\omega}$ and
$(r'_n)_{n\in\omega }$ in $A$ such that $((r_n)_{n\in\omega },
(r'_n)_{n\in\omega }, x)$ is tree-suitable. Then, by Fact \ref{fact:ts}, $T$
is well-founded if and only if $\Phi_{r_nr'_n}(T)\in \V_A$.

(2) Assume first that $A$ is a closed subset of $ \RR^+$, or that $A$ is
Borel and $0$ is not a limit point of $A$. Then, in either case, for any
$X\in F ( \mathbb U )$ we have $X\in \V_A\Leftrightarrow\forall n,m\in\omega\
d(\psi_n(X),\psi_m(X))\in A$: for the backwards implication when $0$ is not a
limit point of $A$, use the fact the the condition $\forall n,m\in\omega\
d(\psi_n(X),\psi_m(X))\in A$ implies that $X$ is discrete, so $X=
\set{\psi_n(X)}{n\in\omega } $.

Conversely, assume that $ \V_A$ is Borel. Then $A$ is Borel by Proposition \ref{basicVA}(1) and if $A$
is not closed and $0$ is a limit point of $A$ we derive a contradiction from
(1).

(3) is immediate from (1) and Fact \ref{upperV}.

To prove (4) first recall that \( \V_A \) is \PI12 by Fact \ref{upperV}.

To show that $\V_A$ is \PI12-hard we fix a strictly decreasing sequence
$(\varepsilon_i)_{i \in \omega}$ in $A$ with $\varepsilon_i < 2^{-i}$.

Let $P \subseteq \Can$ be an arbitrary \PI12 set. We assume first that $A
\subseteq [0,1]$. Since $A$ is \SI11-complete, there exists a continuous
function $f: \Can \times \Can \to [0,1]$ such that $P = \set{\alpha \in
\Can}{\forall \beta \in \Can\, f(\alpha,\beta) \in A}$. Let us define a Borel
function $g$ from $\Can$ to the space of pruned trees on $\{0,1\}$ as
follows. Given $\alpha \in \Can$ consider the compact set $C_\alpha =
\set{f(\alpha,\beta)}{\beta \in \Can} \cup \set{\varepsilon_i}{i \in \omega}
\cup \{0\}$. Define $g(\alpha)$ to be the pruned tree such that $[g(\alpha)]
= \set{\gamma \in \Can}{\sum \gamma(i) 2^{-(i+1)} \in C_\alpha} $.
The function
$g$ is Borel, using the fact that, given $s\in 2^{<\omega }$,
\[
s\in g(\alpha ) \iff
\left [ \sum_{i=0}^{ \lh (s)-1} \frac{s(i)}{2^{i+1}}, \sum_{i=0}^{ \lh (s)-1} \frac{s(i)}{2^{i+1}} + \frac 1{2^{ \lh (s)}} \right ]
\cap C_{\alpha }\neq \emptyset
\]
together with the continuity of the function $ \Can \to
K([0,1]),\alpha\mapsto C_{\alpha }$ (using \cite[Exercise 4.29, iv) and
vi)]{Kechris1995}) and the fact that the relation of non-disjointness is
closed in $(K([0,1]))^2$ (\cite[Exercise 4.29, iii)]{Kechris1995}).

We now apply to $g(\alpha)$ the construction used by Clemens in his proof of
\cite[Theorem 4.7]{ClemensThesis} using the sequence $(\varepsilon_i)_{i \in
\omega}$ we fixed in advance. We thus obtain a function $h: \Can \to
\V_{[0,1]}$ such that $D(h(\alpha)) =C_{\alpha }$ for every $\alpha \in
\Can$. To see that $h$ is Borel one needs to inspect Clemens' construction,
keeping in mind Remark \ref{rem:coding} since $h(\alpha)$ is introduced by
defining the restriction of the distance function to a countable dense
subset. It is immediate that $\alpha \in P$ if and only if $h(\alpha) \in
\V_A$. Thus \( \V_A \) is \PI12-complete.

When $A \subseteq [0,n]$, by rescaling we obtain a function $h_n: \Can \to
\V_{[0,n]}$ reducing $P$ to $\V_A$.

If $A$ is unbounded, let $\varphi : \Can \to \RR^+$ be a continuous function
reducing a \SI11-complete subset of $ \Can $ to $A$. Since the range of
$\varphi$ is bounded, it follows that $A_n = A \cap [0,n]$ is \SI11-complete
for some $n$ and $h_n$ still does the job.
\end{proof}

If we further assume that $A \in \D$ we can draw the following corollaries.
These, combined with Theorem \ref{vaborel}(3), provide a complete picture of
the complexity of $\V_A$ for $A \in \D$ under analytic determinacy (which
ensures that every analytic set which is not Borel is \SI11-complete).

\begin{corollary}\label{VAD}
Let $A \in \D$. The following are equivalent:
\begin{enumerate}[(i)]
  \item $\V_A$ is Borel;
  \item $\V_A$ is \SI11;
  \item $A$ is closed or $0$ is not a limit point of $A$.
\end{enumerate}
\end{corollary}
\begin{proof}
(i) implies (ii) is obvious. (ii) implies (iii) follows from Theorem
\ref{vaborel}(1) because if $\V_A$ is \SI11 then it is not \PI11-hard. To
check that (iii) implies (i) notice that if $A$ is closed then Theorem
\ref{vaborel}(2) applies, yielding immediately (i). If instead $0$ is not a
limit point of $A$ then, since $A \in \D$, $A$ is countable by Theorem
\ref{clemensrealized}. Thus $A$ is Borel and Theorem \ref{vaborel}(2) applies
again.
\end{proof}

\begin{corollary}\label{VAinD}
Let $A \in \D$.
\begin{enumerate}
  \item If $A$ is not Borel then $\V_A$ is neither analytic nor coanalytic;
  \item if $A$ is \SI11-complete then \( \V_A \) is \PI12-complete.
\end{enumerate}
\end{corollary}
\begin{proof}
(1) If $\V_A$ were analytic then it would be Borel by Corollary \ref{VAD},
and then $A$ would be Borel by Proposition \ref{basicVA}(1). Since $A$ is
analytic by Theorem \ref{clemensrealized}, Proposition \ref{basicVA}(1)
implies also that $\V_A$ is not coanalytic.

(2) Since any \SI11-complete set is uncountable, the result follows immediately
from Theorems \ref{clemensrealized} and \ref{vaborel}(4).
\end{proof}

\begin{table}
\begin{tabular}{|c|c|c|} \hline
\textbf{Properties of $A$} & \textbf{complexity of $\V_A$} & \textbf{Reference}\\
\hline\hline
$A$ closed or $0$ isolated in $A$ & Borel & \ref{VAD}\\
\hline
$A$ Borel not closed and $0$ not isolated in $A$ & \PI11-complete & \ref{vaborel}(3)\\
\hline
$A$ true analytic & neither \SI11 nor \PI11 & \ref{VAinD}(1)\\
\hline
$A$ \SI11-complete & \PI12-complete & \ref{VAinD}(2)\\
\hline
\end{tabular}\medskip
\caption{Summary of the complexity of $\V_A$ for $A \in \D$\label{tableM}}
\end{table}
For the reader's convenience we summarize in Table \ref{tableM} our results
for the complexity of $\V_A$ when $A \in \D$.

We now show that the complexity of $\V_A^\star$ often depends on the limit
points of $A$.

\begin{theorem}\label{vastarborel}
Let $A\in \D $.
\begin{enumerate}
  \item $\V_A^\star$ is Borel if and only if either $0$ is not a limit
      point of $A$ or $0$ is the unique limit point of $A$.
  \item Suppose $0$ is a limit point of $A$ and $A$ has other limit points
      (which may belong to $A$ or not).
\begin{enumerate}
\item If $A$ is closed then $ \V_A^\star$ is \SI11-hard;
\item if $A$ is not closed then $ \V_A^\star$ is \PI11-hard;
\item if $A$ is not closed and at least one of its limit points different
    from $0$ belongs to $A$, then $ \V_A^\star$ is $D_2(
    \SI11)$-hard.
\end{enumerate}
  \item if $A$ is \SI11-complete then \( \V_A^\star \) is
      \PI12-complete.
\end{enumerate}
\end{theorem}
\begin{proof}
We start from (2). For (a) pick $y \neq 0$ which is a limit point of $A$:
obviously $y \in A$. Fix now sequences $(s_n)_{n\in\omega
},(s'_n)_{n\in\omega }$ in $A$ such that the triple $((s_n)_{n\in\omega },
(s'_n)_{n\in\omega }, y)$ is tree-suitable.
As $A\setminus\{ y\}\in \mathcal D $ by Theorem \ref{clemensrealized}, fix also a space $Y \in
\V^\star_{A\setminus\{ y\} }$. Then, using Facts \ref{fact:ts} and
\ref{fact:oplus}, the function $T \mapsto \Phi_{s_ns'_n}(T) \oplus_{s_0} Y$
is a Borel reduction from ill-founded trees to $ \V_A^\star$.

(b) follows from Proposition \ref{basicVA}(2) and Theorem \ref{vaborel}(1).

To prove (c) let $x \in \bar A \setminus A$, and let $y \in A \setminus
\{0\}$ which is a limit point of $A$. Fix sequences $(r_n)_{n\in\omega }$,
$(s_n)_{n\in\omega }$, $(t_n)_{n\in\omega }$ in $A\setminus\{ y\}$ such that
both $((r_n)_{n\in\omega }, (s_n)_{n\in\omega }, x)$ and $((r_n)_{n\in\omega
}, (t_n)_{n\in\omega }, y)$ are tree-suitable. By Theorem
\ref{clemensrealized} again, $A\setminus\{ y\}\in \mathcal D $, so fix $X \in
\V^\star_{A \setminus\{ y\}}$. Then, by Facts \ref{fact:ts} and
\ref{fact:oplus}, the Borel function
\[
(U,T) \mapsto \Theta (U,T)=(\Phi_{r_ns_n}(U) \oplus_{r_0} \Phi_{r_nt_n}(T)) \oplus_{r_0} X
\]
is such that $\Theta (U,T)\in \V_A^\star$ if and only if $U$ is well-founded
and $T$ is ill-founded.\smallskip

We now deal with (1). The forward direction follows from (2), since if the
thesis were false one of (a) and (b) would apply.

Conversely, assume first that $0$ is not a limit point of $A$. Then, by Theorem \ref{clemensrealized}, $A$ is
countable and all members of $ \V_A^\star$ are discrete. Then for any $X\in F
( \mathbb U )$ we have that $X\in \V_A^\star$ if and only if $X \in \V_A
\land \forall a \in A\; \exists m_1, m_2 \in \omega\;
d(\psi_{m_1}(X),\psi_{m_2}(X)) = a$. Theorem \ref{vaborel}(2) allows to
conclude in this case.

Finally, suppose $0$ is the unique limit point of $A$. Then $A$ is closed and
countable and all  elements of $A$ different from $0$ are isolated in $A$.
Thus for any $X\in F ( \mathbb U )$ we have again that $X \in \V_A^\star$ if
and only if $X \in \V_A \land \forall a \in A\; \exists m_1, m_2 \in \omega\;
d(\psi_{m_1}(X),\psi_{m_2}(X)) = a$, which allows to conclude by applying
Theorem \ref{vaborel}(2).\smallskip

(3) follows immediately from Proposition \ref{basicVA}(2) and Corollary
\ref{VAinD}(2).
\end{proof}

We now consider the case when $A$ is countable. Then $\V_A$ is either Borel
or \PI11-complete according to Theorem \ref{vaborel}(2 and 3). We can obtain
a complete classification of the complexity of $\V_A^\star$ as well.

\begin{theorem}\label{complcount}
Let $A$ be a countable subset of $ \RR^+$, with $0\in A$.
\begin{enumerate}
\item If $0$ is not a limit point of $A$ or $0$ is the unique limit point
    of  $A$, then $ \V_A^\star$ is Borel;
\item if $0$ is a limit point of $A$ and $A$ is closed having other limit
    points besides $0$, then $ \V_A^\star$ is \SI11-complete;
\item if $0$ is a limit point of $A$, $A$ is not closed, and all limit
    points of $A$ different from $0$ do not belong to $A$, then $
    \V_A^\star$ is \PI11-complete;
\item if $0$ is a limit point of $A$, $A$ is not closed and contains a
    limit point different from $0$, then $ \V_A^\star$ is $D_2(
    \SI11)$-complete.
\end{enumerate}
\end{theorem}
\begin{proof}
(1) follows from Theorem \ref{vastarborel}(1).

(2) In this case, by Theorem \ref{vaborel}(2), $\V_A$ is Borel. Since $X \in
\V^\star_A$ if and only if $X \in \V_A \land \forall a \in A\, \exists x,y
\in X\, d(x,y) = a$, we have that $\V^\star_A$ is \SI11. Completeness follows
from Theorem \ref{vastarborel}(2a).

(3) In this case, all points of $A$ different from $0$ are isolated in $A$ so
that, as at the end of the proof of Theorem \ref{vastarborel}(1), we have $X
\in \V_A^\star$ if and only if $X \in \V_A \land \forall a \in A\, \exists
m_1,m_2\, d(\psi_{m_1}(X),\psi_{m_2}(X))=a$. Hence $\V_A^\star$ is \PI11
because $\V_A$ is \PI11 by Fact \ref{upperV}. Completeness follows from
Theorem \ref{vastarborel}(2b).

(4) By Fact \ref{upperV} $\V^\star_A$ is $D_2( \SI11)$ and completeness
follows from Theorem \ref{vastarborel}(2c).
\end{proof}

\begin{remark}
In the literature, there are very few ``natural'' examples  of  sets
belonging to the class $D_2( \SI11)$ but not to simpler ones. The set \(
\V_A^\star \) with \( A \) as in Theorem \ref{complcount}(4) is one of these. Other notable examples
are: the collection of countable graphs whose automorphism group is isomorphic to ${}^{\omega } \ZZ_p$ for $p$ a prime number \cite{ck2000}; the collection of countable linear orders which are not strongly
surjective~\cite{cammarcar}; and some collection of measurable sets generated
using the density function on the Cantor space \cite{AC}.
\end{remark}

\begin{table}
\begin{tabular}{|c|c|c|} \hline
\textbf{Properties of $A$} & \textbf{complexity of $\V^\star_A$} & \textbf{Reference}\\
\hline\hline
\makecell{$0$ isolated in $A$\\ or $0$ unique limit point of $A$} & Borel & \ref{vastarborel}(1)\\
\hline
\makecell{$0$ not isolated in $A$,\\ $A$ closed with other limit points} & \SI11-hard & \ref{vastarborel}(2)(a)\\
\hline
\makecell{$0$ not isolated in $A$,\\ $A$ \emph{countable} closed with other limit points} & \SI11-complete & \ref{complcount}(2)\\
\hline
\makecell{$0$ not isolated in $A$,\\ $A$ not closed with other limit points} & \PI11-hard & \ref{vastarborel}(2)(b)\\
\hline
\makecell{$0$ not isolated in $A$, $A$ \emph{countable} not closed\\ with all other limit points not in $A$} & \PI11-complete & \ref{complcount}(3)\\
\hline
\makecell{$0$ not isolated in $A$,\\ $A$ not closed with other limit points in $A$} & $D_2( \SI11)$-hard & \ref{vastarborel}(2)(c)\\
\hline
\makecell{$0$ not isolated in $A$, $A$ \emph{countable} not closed\\ with other limit points in $A$} & $D_2( \SI11)$-complete & \ref{complcount}(4)\\
\hline
$A$ \SI11-complete & \PI12-complete & \ref{vastarborel}(3)\\
\hline
\end{tabular}\medskip
\caption{Summary of the complexity of $\V^*_A$ for $A \in \D$\label{tableM*}}
\end{table}
Table \ref{tableM*} summarizes our results for the complexity of $\V^\star_A$
when $A \in \D$.\medskip

Several sections of Gao and Shao's paper \cite{GaoShao2011} are devoted to
different ways of constructing, for every countable $A \in \D$ (which, by
Theorem \ref{necessary}(2), means for every $A$ such that $\U^\star_A \neq
\emptyset$), a Polish $A$-ultrametric spaces $X$ which is $A$-universal
(i.e.\ such that $X \in \U^\star_A$, and $Y \sqsubseteq X$ for every $Y \in
\U_A$) and ultrahomogeneous. They call such a space $A$-ultrametric Urysohn.
The analogous question for Polish spaces was considered by Sauer in
\cite{Sauer} (beware that Sauer calls homogeneous the spaces we call
ultrahomogeneous, and that his definition of universality is equivalent to
ours only for ultrahomogeneous and complete spaces).

\begin{defin}
We say that a metric space $X$ is \emph{ultrahomogeneous} if every isometry between finite subsets of $X$ can be extended to an isometry of the whole $X$.
\end{defin}

\begin{defin}
Let $A \in \D$. We say that $X \in \V_A$ is \emph{Polish $A$-universal} if $Y
\sqsubseteq X$ for every $Y \in \V_A$ (clearly $X \in \V^\star_A$ must then
hold). If additionally $X$ is ultrahomogeneous then we say it is \emph{$A$-Urysohn}.
\end{defin}

Here we use Corollary \ref{VAD} to extend Sauer's characterization
(\cite[Theorem 4.13]{Sauer}, which is the equivalence between (i) and (iv) in
Theorem \ref{Urysohn} below) and give a different proof of the necessity of
the condition for the existence of $A$-Urysohn spaces. The following property
was isolated in \cite{DLPS}.

\begin{defin}
A triple $(a, b, c)$ of elements of $\RR^+$ is \emph{metric} if $a \leq b + c$, $b
\leq a + c$, and $c \leq a + b$. A set $A \subseteq \RR^+$ satisfies the
\emph{$4$-values condition} if for all pairs of metric triples of numbers in
$A$ of the form $(a, b, x)$ and $(c, d, x)$ there exists $y \in A$ such that
both $(b, c, y)$ and $(a, d, y)$ are metric triples.
\end{defin}

\begin{theorem}\label{Urysohn}
Let $A \in \D$. The following are equivalent:
\begin{enumerate}[(i)]
  \item $A$ satisfies the $4$-values condition and either is closed or $0$
      is not a limit point of $A$;
  \item $A$ satisfies the $4$-values condition and $\V_A$ is Borel;
  \item $A$ satisfies the $4$-values condition and $\V_A$ is \SI11;
  \item there exists $X \in \V_A$ which is $A$-Urysohn;
  \item $A$ satisfies the $4$-values condition and there exists $X \in
      \V_A$ which is Polish $A$-universal.
\end{enumerate}
\end{theorem}
\begin{proof}
The equivalence between (i), (ii), and (iii) follows immediately from
Corollary \ref{VAD}.

(i) implies (iv) is obtained by Sauer repeating the classical construction of
the Urysohn space (which in our terminology would be $\RR^+$-Urysohn) using
only spaces in $\V_A$: we amalgamate (using the $4$-values condition) the
finite members of $\V_A$ obtaining an ultrahomogeneous $Z$ with $D(Z) = A$
which isometrically embeds all countable metric spaces using distances in
$A$. We then let $X$ to be the completion of $Z$, and we need to check that
(i) guarantees that $X$ does not use distances outside $A$. If $0$ is not a
limit point of $A$ then $Z$ is discrete so that $X=Z \in \V_A$. Otherwise
$D(X) \subseteq \overline{D(Z)} = \overline{A}$, thus if $A$ is closed we
have $X \in \V_A$. Moreover $X$ is still ultrahomogeneous by \cite[Theorem
4.4]{Sauer}.

To prove (iv) implies (v) we need Theorem 3.9 of \cite{Sauer}, stating that
the set of distances of an ultrahomogeneous universal metric space satisfies
the $4$-values condition.

We complete the proof by showing that (v) implies (iii). This is immediate
because if a Polish $A$-universal $X$ exists, then $Y \in \V_A$ if and only
if $Y \sqsubseteq X$, and $\sqsubseteq $ is analytic.
\end{proof}

\section{Isometry and isometric embeddability}\label{polish}

As a first step in our analysis of isometry and isometric embeddability
restricted to $\V_A$ and $\V^\star_A$ we prove the following proposition,
which also answers the first part of \cite[Question 3]{ClemensPreprint} by
showing that if \( A,A' \in \D \) and \( A \subseteq A' \), then \(
{\isom^\star_A} \leq_B {\isom^\star_{A'} } \).

\begin{proposition} \label{questionclemens}
Let $A,A' \in \D$ and assume $A \subsetneq A'$. Then ${\isom_A^\star} \leq_B
{\isom_A} \leq_B {\isom_{A'}^\star} \leq_B {\isom_{A'}}$ and
${\sqsubseteq_A^\star} \leq_B {\sqsubseteq_A} \leq_B {\sqsubseteq_{A'}^\star}
\leq_B {\sqsubseteq_{A'}}$.
\end{proposition}

\begin{proof}
For any $A$ we have $\V^\star_A \subseteq \V_A$ and hence ${\isom_A^\star}
\leq_B {\isom_A}$ and ${\sqsubseteq^\star_A} \leq_B {\sqsubseteq_A}$. Thus we
need only to prove ${\isom_A} \leq_B {\isom_{A'}^\star}$ and ${\sqsubseteq_A}
\leq_B {\sqsubseteq_{A'}^\star}$.

Suppose first that $A'=A\cup\{ r_0\} $ and $a<r_0$ for all $a\in A$. Fix a
space $Z \in \V_A^\star$. Given $X\in \V_A$, define $X'= X \oplus_{r_0} Z \in
\V^\star_A$ by Fact \ref{fact:oplus}: this map is Borel by Proposition
\ref{1531426}. We show that it is the required reduction. Notice that by case
assumption $d_{X'}(x,z)=r_0$ whenever $x\in X$ and $z\in Z$. If $\psi\colon
X\to Y$ is an isometric embedding (respectively, an isometry), then $\psi\cup
{\rm id}_Z \colon X'\to Y'$ is an isometric embedding (respectively, an
isometry) as well. Conversely, suppose $\psi \colon X'\to Y'$ is an isometric
embedding. Since the distance \( r_0 \) is never realized inside any of \( X
\), \( Y \), and \( Z \), either $\psi(X)\subseteq Y$ and $\psi(Z)\subseteq
Z$, in which case $\psi \restriction X$ witnesses $X \sqsubseteq Y$, or else
$\psi(X)\subseteq Z$ and $\psi(Z)\subseteq Y$. In the latter case, $X
\sqsubseteq Z$ and $Z \sqsubseteq Y$, so again $X \sqsubseteq Y$. If moreover
$\psi$ is onto, then we have either $\psi(X)=Y$, or else $\psi(X)=Z$ and
$\psi(Z)=Y$, which allows us to conclude that $X \isom Y$.

Otherwise there exist $r_0 \in A' \setminus A$ and $r_1 \in A'$ with
$r_0<r_1$. Notice that there is $Z\in \V_{A'}^\star$ with the property that
$d_Z(z_0,z_1)=r_0$ for exactly one pair of  points $\{ z_0,z_1\} \subseteq
Z$. Indeed $A'\setminus\{ r_0\}\in \mathcal D $ by Theorem \ref{clemensrealized}, so let $Z = W \oplus_{r_1} V$ where $W \in \V_{A' \setminus \{r_0\}
}^\star$ and $V = \{z_0,z_1\}$ consists of two points at distance $r_0$, so
that $d_Z(w,z_i) \geq r_1$ for every $w \in W$ and $i \in \{0,1\}$.

We show that the mapping that associates to every $X\in \V_A$ the space
$X\times Z$ with the product metric $d_{X\times Z}((x,z),(x',z'))=\max
\{d_X(x,x'),d_Z(z,z')\}$ is the required Borel reduction. Borelness is proved
using Remark \ref{rem:coding}.

If $\psi \colon X\to Y$ is an isometric embedding (respectively, an
isometry), then $\psi\times {\rm id}_Z \colon X\times Z\to Y\times Z$ is an
isometric embedding (respectively, an isometry). Conversely, suppose $\psi
\colon X\times Z\to Y\times Z$ is an isometric embedding. Then for every
$x\in X$ there exists $y\in Y$ such that $\psi(x,z_0)\in\{ (y,z_0),(y,z_1)\}
$, since in both $X\times Z$ and $Y\times Z$ the points having second
coordinate equal to $z_0$ or $z_1$ are the only points that realize the
distance $r_0$ by our choice of $Z$. This defines a function $\varphi \colon
X\to Y$ with the property that for every $x\in X$ we have
$\psi(x,z_0)=(\varphi(x),z_i)$ for some $i\in\{ 0,1\} $. In order to prove
that $\varphi$ is an isometric embedding, let $x,x'\in X$ and let
$r=d_X(x,x')=d_{X\times Z}((x,z_0),(x',z_0))=d_{Y\times
Z}(\psi(x,z_0),\psi(x',z_0))$. Since we have $\psi(x,z_0)=(\varphi(x),z_i)$
and $\psi(x',z_0)=(\varphi(x'),z_j)$ for some $i,j\in\{ 0,1\} $, we obtain $r
= \max \{d_Y(\varphi(x),\varphi(x')),d_Z(z_i,z_j)\}$; recalling that $r\neq
r_0 = d_Z(z_0,z_1)$, the equality $d_Y(\varphi(x), \varphi(x'))=r$ follows.
So $\varphi$ is an isometric embedding. Assume now $\psi$ is surjective and
let $y\in Y$. So, suppose $\psi(x,z)=(y,z_0)$ and $\psi(x',z')=(y,z_1)$.
Recall that $r_0\notin A$ is not realized in $X\in \V_A$ and is realized only
by the pair $\{ z_0,z_1\} $ in $Z$. Since $d_{Y \times
Z}((y,z_0),(y,z_1))=r_0$ we have that $\{ z,z'\} =\{ z_0,z_1\} $. If $z=z_0$
then $\varphi(x)=y$, while if $z'=z_0$ then $\varphi(x')=y$. So $\varphi$ too
is surjective.
\end{proof}

\begin{remark}\label{remark}
One may be interested in analogues of Proposition~\ref{questionclemens}
obtained by restricting the relations of isometry and isometric embeddability
to a given class of Polish metric spaces.

The same proof shows that the conclusion of the proposition holds for
ultrametric, zero-dimensional, countable, locally compact, $\sigma $-compact,
and discrete spaces. This is the case because such classes are closed under
finite products and the operations $\oplus_r$, and whenever they have an
element in \( \V^\star_{A'} \), for every $r_0 \in A'$ they also have an
element in $ \V_{A'\setminus\{ r_0\} }$ and contain a space consisting of two
points at distance $r_0$. For the classes mentioned above the latter property
follows from Theorem \ref{necessary}: this is clear for ultrametric,
zero-dimensional, countable, and discrete spaces; for the classes of locally
compact and $\sigma $-compact spaces notice that removing a point from a
$\sigma $-compact subset of $\RR$ yields a $\sigma $-compact set. For
contrast, this argument does not work for compact metric spaces\footnote{In a
previous version of the paper we claimed that this was the case: we thank the
referee for pointing out our mistake.} and we do not know whether
Proposition~\ref{questionclemens} holds restricted to this class.

One can also restrict attention to spaces of a fixed dimension different from
$0$ and obtain the same results even if these classes are not closed under
finite products: this is because in the last paragraph of the proof of
Proposition~\ref{questionclemens} we can require (by Theorem
\ref{necessary}(1)) $Z$ to be zero-dimensional, so that $X \times Z$ has the
same dimension of $X$.
\end{remark}

We will use the following folklore construction to turn a countable graph
into a discrete metric space. Fix $r,r' \in \RR$ with $r<r'\leq 2r$. To each
graph $G$ on $\omega$ associate the metric space $X_G=(G,d_G)$ by letting
$d_G(a,b)=r$ if $(a,b)$ is an edge in $G$, and $d(a,b)=r'$ if $a\neq b$ and
$(a,b)$ is not an edge in $G$. The following Lemma is straightforward.

\begin{lemma}\label{graphs}
The map $G \mapsto X_G$ Borel reduces countable graph isomorphism and
countable graph embeddability to $\isom_{\{0,r,r'\}}$ and
$\sqsubseteq_{\{0,r,r'\}}$, respectively. Moreover, if we restrict the map to
nontrivial graphs (i.e.\ different from the empty graph and from the
countable clique), we get a reduction to $\isom^\star_{\{0,r,r'\}}$ and
$\sqsubseteq^\star_{\{0,r,r'\}}$.
\end{lemma}

\subsection{Isometry}
The study of the complexity of $\isom_A^\star$ was started by Clemens in
\cite{ClemensPreprint}, where such relation is called $E_A$. Clemens' main
results about $ \isom_A^\star$ are summarized in the following theorem.

\begin{theorem}[{\cite[Theorem 23]{ClemensPreprint}}] \label{thmclemens}
Let \( A \in \D \).
\begin{enumerate}
\item If $A$ contains a right neighborhood of $0$, then $ \isom_A^\star$ is
    Borel bireducible with any complete orbit equivalence relation.
\item If $A$ is dense in some right neighborhood of $0$ but does not
    contain any of them, then both the action of the density ideal on $
    \pre{\omega }{2} $ and countable graph isomorphism are Borel reducible
    to $ \isom_A^\star$. In particular, $ \isom_A^\star$ is strictly above
    countable graph isomorphism with respect to \(\leq_B \).
\item If $A$ is not dense in any right neighborhood of $0$ and either $0$
    is a limit point of $A$ or \( A \) is not well-spaced, then $
    \isom_A^\star$ is Borel bireducible with countable graph isomorphism.
\item If $A = \{ 0\} \cup \{ r_i \mid i \in \omega \}$ with $0<r_i<r_{i+1}$
    is well-spaced, then \( \isom^\star_A \) is Borel bireducible with
    isomorphism between reverse trees (as defined in
    \cite{ClemensPreprint}).
\item If $A=\{ 0,r_0,\ldots ,r_{n-1}\}$ with $0<r_i<r_{i+1}$ is
    well-spaced, then $ \isom_A^\star$ is Borel bireducible with
    isomorphism between trees of height $n$. Thus these relations form a
    $\leq_B$-strictly increasing chain of equivalence relations
    classifiable by countable structures and they are strictly below
    countable graph isomorphism.
\end{enumerate}
\end{theorem}

Notice that the conditions considered in Theorem \ref{thmclemens} are
exhaustive. In fact, if $0$ is not isolated in $A$ or $A$ is not well-spaced
then we are in either case (1), (2), or (3). If $0$ is isolated in $A$ and \(
A \) is well-spaced then $A$ is well-founded, since strictly decreasing sequences in a
well-spaced set converge to $0$. Since a well-founded and well-spaced set
has order type \( \leq \omega \), we are either in case (4) or in case (5).

Proposition~\ref{questionclemens} may be used to give a simpler proof of part
(1) of Theorem~\ref{thmclemens}. Let $r$ be such that $[0,r]\subseteq A$. To
any Polish metric space $(X,d)$, associate the space $(X,d')$, where
$d'(x,y)=r \cdot \frac{d(x,y)}{1+d(x,y)} $. This reduces isometry on all
Polish metric spaces to $ \isom_{[0,r)}$, which in turn reduces to $
\isom_A^\star$ by Proposition~\ref{questionclemens}. Since isometry on
arbitrary Polish metric spaces is Borel bireducible with the complete orbit
equivalence relation by \cite[Theorem 1]{Gao2003}, we are done.

Theorem \ref{thmclemens} yields the following sufficient condition for
countable graph isomorphism being Borel reducible to \( \isom^\star_A \).

\begin{corollary} \label{corClemens}
Let \( A \in \D \). If \( A \) is ill-founded or not well-spaced, then
countable graph isomorphism Borel reduces to \( \isom^\star_A \).
\end{corollary}

We will now show how the results from \cite{ultrametric} cited in Section
\ref{term} can be used to complete the description of the behaviour of
$\isom_A^\star$. We begin by proving the converse of Corollary
\ref{corClemens}, i.e.\ we characterize when countable graph isomorphism is
Borel reducible to \( \isom^\star_A\).

\begin{theorem} \label{appendixisom}
Let $A \in \D$. Countable graph isomorphism Borel reduces to $ \isom_A^\star$
if and only if $A$ is either ill-founded or not well-spaced.
\end{theorem}
\begin{proof}
One direction is Corollary~\ref{corClemens}, but for the reader's convenience
we give here an alternative and simpler proof. If $A$ is ill-founded, let $(
r_n)_{n\in\omega }$ be a decreasing sequence in $A$. Since countable graph
isomorphism Borel reduces to $ \isom_{( r_n)_{n \in \omega}}^\star$ by
Theorem \ref{isomuniversalstar}(1), it suffices to use
Proposition~\ref{questionclemens} when $( r_n)_{n\in\omega } \subsetneq A$.
If instead $A$ is not well-spaced fix $r,r'\in A$ with $r<r'\leq 2r$. Lemma
\ref{graphs} gives a Borel reduction of isomorphism between nontrivial
countable graphs to $ \isom^\star_{\{ 0,r,r'\} }$. Then apply
Proposition~\ref{questionclemens} if \( \{ 0,r,r' \} \subsetneq A \).

Finally, assume that $A$ is well-founded and well-spaced. Then $ \V_A^\star=
\U_A^\star$ by Theorem \ref{proponlyultrametric}(2), and countable graph
isomorphism does not Borel reduce to $ \isom_A^\star$ because the latter is
Borel by Theorem \ref{lemmaequivalence}(2).
\end{proof}

We now have the following fairly complete picture of the structure of the
relations $ \isom^\star_A$. Notice that conditions (1)--(4) exhaust all
possible cases for $A$.

\begin{theorem}\label{isomstarA}
Let $A \in \D$.
\begin{enumerate}
\item The relations \( \isom^\star_A \) with \( A \) well-founded and
    well-spaced form a strictly increasing chain of order type \( \omega +
    1 \) under $\leq_B$, consisting of Borel equivalence relations, and
    they are Borel reducible to all the other \( \isom^\star_{A'} \) with
    $A' \in \D$.
\item If $A$ is either ill-founded or not well-spaced, and moreover $A$ is
    not dense in any right neighborhood of $0$, then $ \isom_A^\star$ is
    Borel bireducible with countable graph isomorphism.
\item If $A$ is dense in some right neighborhood of $0$ but does not
    contain any such neighborhood, then $\isom_A^\star$ is strictly above
    countable graph isomorphism and Borel reducible to any complete  orbit
    equivalence relation.
\item If $A$ contains a right neighborhood of $0$, then $ \isom_A^\star$ is
    Borel bireducible with any complete  orbit equivalence relation.
\end{enumerate}
\end{theorem}

\begin{proof}
Parts (2)--(4) follow from Theorem~\ref{thmclemens} and the observations
following it, so let us  prove (1).

Let $\alpha$ be the order type of $A$. The fact that $A$ is well-spaced
implies \( 1 \leq \alpha \leq \omega \) and $\V^\star_A = \U^\star_A$ (by
Theorem \ref{proponlyultrametric}(2)). Thus ${\isom^\star_A}$ and
${\isom^\star_{\alpha}}$ are the same relation. By Theorem \ref{lemmaequivalence}(2), when \( \alpha
\) varies between \( 1 \) and \( \omega \) these equivalence relations form a
strictly increasing chain of length $\omega+1$ under $\leq_B$ and they are
Borel equivalence relations. Finally, since \( \isom^\star_{\alpha} \) Borel
reduces to countable graph isomorphism by Theorem \ref{lemmaequivalence}(1),
it follows from Theorem~\ref{appendixisom} that \( \isom^\star_A \) reduces
to any other \( \isom^\star_{A'} \) for \( A' \) not satisfying the
conditions of (1).
\end{proof}

We will now partially answer also the second part of \cite[Question
3]{ClemensPreprint}, which asked whether ${\isom_A} \sim_B {\isom^\star_A}$
for every \( A \in \D \).

\begin{theorem}\label{thm:isomA}
Let $A \in \D$ satisfy at least one of the following conditions:
\begin{enumerate}[(i)]
  \item $A$ is not dense in any right neighborhood of $0$;
  \item $A$ has maximum;
  \item there exists $f \colon A \to A$ which is Polish metric preserving,
      injective, and non-surjective.
\end{enumerate}
Then ${\isom_A} \sim_B {\isom^\star_A}$.
\end{theorem}
\begin{proof}
For any $A \in \D$ we have ${\isom^\star_A} \leq_B {\isom_A}$ because
$\V^\star_A \subseteq \V_A$.\smallskip

Assume first that $A$ is not dense in any right neighborhood of $0$. In this
case ${\isom_A}$ is classifiable by countable structures. In fact the
argument of \cite[Proposition 18]{ClemensPreprint} applies not only to
${\isom^\star_A}$ but also to ${\isom_A}$. We distinguish two cases. If $A$
is either ill-founded or not well-spaced, then ${\isom^\star_A}$ is Borel
bireducible with countable graph isomorphism by Theorem \ref{isomstarA}(2),
and hence ${\isom_A} \leq_B {\isom^\star_A}$. If instead $A$ is well-founded
and well-spaced then $\V_A = \U_A$ and $\V^\star_A = \U^\star_A$, so that we
can use Theorem \ref{lemmaequivalence}(1).\smallskip

Now assume that $r = \max A$. Since $ \V_{\{ 0\} }= \V_{\{ 0\} }^\star$, we
may assume that $A$ has more than one element. Fix $Z \in \V^\star_{A
\setminus \{r\}}$ and consider the map sending $X \in \V_A$ to $X' = X
\oplus_r Z \in \V^\star_A$, so that by case assumption $d_{X'}(x,z) =r$
whenever $x \in X$ and $z \in Z$. The map is Borel and we claim that it
witnesses ${\isom_A} \leq_B {\isom^\star_A}$. To show this we use an argument
similar to the one employed in the first part of the proof of Theorem
\ref{lemmaequivalence}(1) in \cite{ultrametric}.
Fix \(
X,Y \in \V_A \). First assume that \( \varphi \colon X \to Y \) witnesses \(
X \isom_A Y \): then \( {\varphi} \cup {\operatorname{id}_Z} \) is a witness
of \( X' \isom^\star_A Y' \).

Conversely, let \( \psi \) be an isometry between \( X' \) and \( Y' \), and
let \( X_0  = \psi^{-1}(Z) \), so that \( \psi (X' \setminus  X_0 ) = Y \).
Notice that \( d_{X'}(x_0,x_1) < r \) for every \( x_0,x_1 \in X_0 \) since
\( \psi(x_0),\psi(x_1) \in Z \in \V^\star_{A \setminus \{r\}} \). Hence, by
construction of \( X' \), either \( X_0 \subseteq Z \) or \( X_0 \subseteq X
\). In the former case \( X_0 = Z \) because $\psi(Z)$ cannot intersect both
$Y$ and $Z$, since any two points of $Z$ are less than $r$ apart. Thus \(
\psi(X) = \psi(X' \setminus X_0) = Y \) and \( \psi \restriction X \) is an
isometry between \( X \) and \( Y \).
If instead $X_0 \subseteq X$, we claim that \( \varphi =  \psi \restriction
(X \setminus X_0) \cup (\psi \circ \psi) \restriction X_0 \) witnesses \( X
\isom_A Y \). Notice that \( \varphi \) is well-defined by the fact that by
definition \( \psi( X_0) = Z \subseteq X' \). For the same reason, the range
of \( \varphi \) equals the range of \( \psi \restriction (X \setminus X_0)
\cup \psi \restriction Z  =  \psi \restriction (X' \setminus X_0) \), thus \(
\varphi \) is a surjection from \( X \) onto \( Y \).
Finally, we check that \( \varphi \) preserves distances. It is clearly
enough to show that for \( x \in X \setminus X_0 \) and \( x' \in X_0 \) we
have \( d_X(x,x') = d_Y(\varphi(x),\varphi(x')) \). Since \( \psi(x) \in Y \)
and \( \psi(x') \in Z \), we have \( d_{Y'}(\psi(x),\psi(x')) = r \), whence
\( d_X(x,x') = d_{X'}(x,x') = d_{Y'}(\psi(x),\psi(x')) = r \). Since \(
\psi(x') \in Z \subseteq X'\) and \( x \in X \), \( d_{X'}(x, \psi(x')) = r
\), therefore \( d_{Y'}(\psi(x), \psi(\psi(x'))) = r \). Since \( \psi(x) =
\varphi(x) \) and \( \psi(\psi(x')) = \varphi(x') \) by definition of \(
\varphi \), we have \( d_Y(\varphi(x),\varphi(x')) =
d_{Y'}(\varphi(x),\varphi(x')) = r = d_X(x,x') \), as required.
\smallskip

Finally, assume that $f$ is as in (iii) and let $A' \subsetneq A$ be the
range of $f$. We map any $X \in \V_A$  to $X' \in \V_{A'}$ by composing the
metric \( d_X \) with $f$: the fact that $f$ is Polish metric preserving
makes sure that \( d_{X'} =  f \circ d_X \) is still a Polish metric space.
Consider the map \( X \mapsto X' \). It is Borel because $f$ is such by
Proposition \ref{prop:automaticBorel}, and, since $f$ is injective, witnesses
${\isom_A} \leq_B {\isom_{A'}}$. By Proposition~\ref{questionclemens} we have
${\isom_{A'}} \leq_B {\isom_A^\star}$, and thus ${\isom_A} \leq_B
{\isom_A^\star}$.
\end{proof}

It is not obvious when Condition (iii) of Theorem \ref{thm:isomA} holds.
Notice that a sufficient condition for a nondecreasing $f \colon A \to \RR$
to be metric preserving is that for all $r,s,t\in A$, if $s \le t < r \le
s+t$ then $f(r)\le f(s)+f(t)$. Using this we see that Condition (iii) holds
for instance when $A = \QQ^+$ or $A = (\RR^+ \setminus \QQ) \cup \{0\}$, as
witnessed by the map $f(r) = r/(1+r)$. The same is true when $A$ contains an
interval $[a,b]$, as witnessed by
\[
f(r) =
\begin{cases}
r & \text{if $r<a$;}\\
\displaystyle{a + (b-a)\frac{r-a}{1+(r-a)}} & \text{if $r \geq a$.}
\end{cases}
\]
(Notice that $b$ does not belong to the range of $f$.) On the other hand,
Condition (iii) can fail even for countable sets: if $A = \{0\} \cup \{2^k
\mid k \in \ZZ\}$ then every injective Polish metric preserving function $f:A
\to A$ satisfies $f(2^k) = 2^{k+z}$ for some $z \in \ZZ$, and hence  is
surjective. Notice however that this particular $A$ satisfies Condition (i)
of Theorem \ref{thm:isomA}.

In fact, we do not know if there exists $A \in \D$ which does not satisfy any
of the conditions of Theorem \ref{thm:isomA}. However if such an $A$ exists,
it must be uncountable. To show this, we need the following fact, which might
be of independent interest.

\begin{proposition}\label{prop:countable}
Let $A \in \D$ be such that $A$ is dense in $(a,b)$ and $A \cap (a, +\infty)$
is countable for some $a,b$ with $0<a<b$. Then Condition (iii) of Theorem \ref{thm:isomA}
holds, i.e.\ there exists $f \colon A \to A$ which is Polish metric
preserving, injective, and non-surjective.
\end{proposition}
\begin{proof}
We may assume that $a,b \in A$. We will define $f$ strictly increasing which
is the identity up to $a$ and maps $A \cap (a, +\infty)$ into $(a,b)$, so
that $b \in A$ is not in the range of $f$.

Let $(a_n)_{n\in \omega}$ be an enumeration without repetitions of $A \cap
(a, +\infty)$. Since $A$ is dense in $(a,b)$ we can recursively define
$f(a_n)$ so that:
\begin{itemize}
  \item $f(a_n)<a_n$;
  \item if $c_0=a < c_1 < \dots < c_{n+1}$ enumerate in increasing order
      $\{a\} \cup \{ a_m \mid m \leq n\}$ then the slope of the segment
      with endpoints $(c_i, f(c_i))$ and $(c_{i+1}, f(c_{i+1}))$ is larger
      than the slope of the segment with endpoints $(c_{i+1}, f(c_{i+1}))$
      and $(c_{i+2}, f(c_{i+2}))$ for each $i<n$.
\end{itemize}

It remains to prove that $f$ is Polish metric preserving. Since $f$ is
nondecreasing, by Proposition \ref{Polish_metric_pres} and the observation
after Theorem \ref{thm:isomA}, it suffices to show that for all $r,s,t\in A$,
if $s \le t < r \le s+t$ then $f(r)\le f(s)+f(t)$. By construction we have
that if $x,y,z \in A$ with $x<y<z$ then $S(y,z) \leq S(x,z) \leq S(x,y)$
where $S(v,w)$ is the slope of the segment with endpoints $(v, f(v))$ and
$(w, f(w))$. From this it follows that for $r,s,t$ as above $S(t,r) \leq
S(0,s)$, whence
\[
f(r) =f(t) + (r-t) \cdot S(t,r) \leq f(t) + s \cdot S(t,r) \leq f(t) + s \cdot S(0,s) = f(t) + f(s).
\]

\end{proof}

\begin{theorem} \label{thm:countable}
If $A \in \D$ is countable then ${\isom_A} \sim_B {\isom^\star_A}$.
\end{theorem}
\begin{proof}
If $A$  is not dense in any right neighborhood of $0$ we are in case (i) of
Theorem \ref{thm:isomA}. Otherwise by Proposition \ref{prop:countable} we are
in case (iii) of Theorem \ref{thm:isomA}.
\end{proof}

\subsection{Isometric embeddability}
Recall again that if $A$ is well-founded and well-spaced, then $ \V_A^\star=
\U_A^\star$ and the order type of $A$ is $\leq \omega$. Hence in this case
the structure of the relations $\sqsubseteq_A^\star$ is described by (i) and
(iv) of Theorem \ref{lemmaequivalence}(4). The remaining case is settled by
the following proposition.

\begin{proposition}\label{prop:completeqo}
Let $A\in \D$. If $A$ is either ill-founded or not well-spaced, then
$\sqsubseteq_A^\star$ is Borel bireducible with a complete analytic
quasi-order.
\end{proposition}
\begin{proof}
First notice that, even if $\V_A^\star$ may not be a standard Borel space,
the relation $\sqsubseteq_A^\star$ still Borel reduces to isometric
embeddability on all Polish metric spaces, which is a complete analytic
quasi-order on a standard Borel space (in fact, Louveau and Rosendal
\cite{louros} showed that isometric embeddability restricted to ultrametric
Polish spaces is a complete analytic quasi-order, and we strengthened this in
\cite{cammarmot} by showing that it has the stronger property of being
invariantly universal).

If $A$ is ill-founded, let $( r_n)_{n\in\omega }$ be a decreasing sequence in
$A$, and let \( A' = \{ 0 \} \cup \{ r_n \mid n \in \omega \} \). Then
isometric embeddability on $\U^\star_{A'}$ is a complete analytic quasi-order
by Theorem \ref{isomuniversalstar}(2). Hence so is $\sqsubseteq^\star_{A'}$
because $\U^\star_{A'} \subseteq \V^\star_{A'}$. By applying
Proposition~\ref{questionclemens} in case \( A' \subsetneq A \), we get the
desired result.

Suppose now that there are $r,r'\in A$ with $r<r'\leq 2r$. Then Lemma
\ref{graphs} yields a Borel reduction of embeddability between nontrivial
graphs on $\omega $ to $\sqsubseteq_{\{ 0,r,r'\} }^\star$. Now apply
Proposition~\ref{questionclemens} again in case \( \{0,r,r'\} \subsetneq A
\).
\end{proof}

Summing up, we have the following full description of the relations \(
\sqsubseteq^\star_A \) for \( A \in \D \).

\begin{theorem}\label{sqsubseteqstarA}
Let $A \in \D$.
\begin{enumerate}
\item The relations \( \sqsubseteq^\star_A \) with \( A \) well-founded and
    well-spaced form a strictly increasing chain of order type \( \omega +
    1 \) consisting of Borel quasi-orders, i.e.\ the quasi-orders \(
    \sqsubseteq^\star_\alpha \) for \( \alpha \leq \omega \) from (i) and
    (iv) of Theorem \ref{lemmaequivalence}(4). These relations are Borel
    reducible to all remaining \( \sqsubseteq^\star_{A'} \) for $A'\in \mathcal D $.
\item The relations \( \sqsubseteq^\star_A \) when \( A \) is either
    ill-founded or not well-spaced are Borel bireducible with a complete
    analytic quasi-order.
\end{enumerate}
\end{theorem}

The following is the analogue of Theorem \ref{thm:isomA}, but in this case
the result is unconditional.

\begin{corollary}\label{corsqsubseteqA}
If \( A \in \D \) then ${\sqsubseteq_A} \sim_B {\sqsubseteq^\star_A}$.
\end{corollary}
\begin{proof}
For any $A$ we have ${\sqsubseteq^\star_A} \leq_B {\sqsubseteq_A}$ because
$\V^\star_A \subseteq \V_A$.

If we are in case (1) of the previous Theorem, ${\sqsubseteq_A} \leq_B
{\sqsubseteq^\star_A}$ follows from Theorem \ref{lemmaequivalence}(3) because
in this case $\V_A = \U_A$ and $\V^\star_A = \U^\star_A$ by Theorem
\ref{proponlyultrametric}(2). If we are in case (2), notice that
${\sqsubseteq_A}$ is Borel reducible to isometric embeddability on arbitrary
Polish spaces. The latter, being an analytic quasi-order on a standard Borel
space, is Borel reducible to any complete analytic quasi-order and hence to
${\sqsubseteq^\star_A}$.
\end{proof}

We observe that Theorem \ref{sqsubseteqstarA} holds even if we further
restrict the relation $\sqsubseteq^\star_A$ to zero-dimensional spaces. In
fact in case (1) all elements of $\V^\star_A$ are discrete and hence
zero-dimensional, while in case (2), which is based on
Proposition~\ref{prop:completeqo}, we use Remark \ref{remark}. For any other
topological dimension we have that $\V^\star_A$ contains such spaces if and
only if $A$ includes a right neighborhood of $0$ (by minor modifications of
the proof of Theorem \ref{proponlyultrametric}(1)). Theorem
\ref{sqsubseteqstarA} holds also in this case by Remark \ref{remark}, but
only case (2) can occur. Similar observations apply to Corollary
\ref{corsqsubseteqA}.

The analogous problems about isometry between spaces of fixed dimension
(different from $\infty$) appear to be more delicate and are discussed more
in detail in \cite[Question 7.1 and the ensuing discussion]{ultrametric}.
The only result about isometry that we know still holds after fixing
dimension is Theorem \ref{thm:isomA}, because we can still use Remark
\ref{remark}.

Proposition \ref{prop:completeqo} can be strengthened by replacing
completeness with invariant universality (the notion originates in
\cite{frimot}, and was formally introduced in \cite{cammarmot}).

\begin{defin}\label{def:invariantlyuniversal}
Let the pair \( (S,E) \) consist of an analytic quasi-order \( S \) and an
analytic equivalence relation \( E \subseteq S \), with both relations
defined on the same standard Borel space $X$. Then $(S,E)$ is
\emph{invariantly universal} (for analytic quasi-orders) if for any analytic
quasi-order $R$ there is a Borel $B\subseteq X$ invariant under $E$ such that
$R\sim_BS\restriction B$.

When $E$ is isometry and $S$ is isometric embeddability on some class of
metric spaces, we just say that $S$ is invariantly universal.
\end{defin}

Notice that if $(S,E)$ is invariantly universal, then $S$ is complete for
analytic quasi-orders.

A notion strictly connected to invariant universality is the following (see
\cite{ultrametric}). Given a pair \( (S,E) \) as above, we denote by \( S/E
\) the \( E \)-quotient of \( S \), i.e.\ the quasi-order on $X/E$ induced by
$S$. If $F$ and $E$ are equivalence relations on sets $X$ and $Y$ and $f
\colon X/F\to Y/E$, then a \emph{lifting} of $f$ is a function $ \hat f
\colon X\to Y$ such that $[ \hat f (x)]_E=f([x]_F)$ for every $x\in X$.

\begin{defin}\label{def:cB}
Let $(R,F)$ and $(S,E)$ be pairs consisting of a quasi-order and an
equivalence relation on some standard Borel spaces, with $F \subseteq R$ and
$E \subseteq S$.

We say that $(R,F)$ is \emph{classwise Borel isomorphic} to $(S,E)$, in
symbols $(R,F) \simeq_{cB} (S,E)$, if there is an isomorphism of quasi-orders
\( f \) between \( R/F \) and \( S/E \) such that both \( f \) and \( f^{-1}
\) admit Borel liftings.


When the equivalence relations $F$ and $E$ are clear from the context we just
say that $R$ is classwise Borel isomorphic to
$S$, and
write \( R \simeq_{cB} S \).
\end{defin}

It is easy to see that if $(R,F)$ is invariantly universal and for some Borel
$E$-invariant $B$ we have $(R,F) \simeq_{cB} (S \restriction B, E
\restriction B)$ then $(S,E)$ is invariantly universal as well.

\begin{lemma} \label{lemma:invuniv1}
Let \( A = \{ 0,r,r' \} \) with \( r < r' \leq 2r \). Then \( \V^\star_A \)
is Borel in \( F(\mathbb{U}) \) and \( \sqsubseteq^\star_A \) is invariantly
universal.
\end{lemma}
\begin{proof}
The set $\V^\star_A$ is Borel by Theorem \ref{vastarborel}. We now show that
embeddability between nontrivial countably infinite graphs is classwise Borel
isomorphic to the restriction of $\sqsubseteq^\star_A$ to the class of
infinite spaces. This suffices because embeddability between these graphs is
invariantly universal by \cite{frimot}. First notice that the class of
infinite spaces in $\V^\star_A$ is Borel because is the collection of \( X
\in \V^\star_A \) satisfying
\[
\forall n \in \omega \, \exists m \in \omega \, \forall i \leq n (\psi_m(X) \neq \psi_i(X)).
 \]
The classwise Borel isomorphism between the quotient quasi-orders is the
quotient of the map introduced before Lemma \ref{graphs}. The inverse of this
map is induced by the Borel  function $X\mapsto G_X$ where $(n,m)$ is an edge
in $G_X$ if and only if $d_X(\psi_n(X),\psi_m(X))=r$.
\end{proof}

\begin{remark} \label{rmk:invuniv1}
The proof of \cite[Theorem 3.9]{frimot} actually shows that the embeddability
relation is already invariantly universal when restricted to the (Borel)
class of connected graphs on \(\omega\): this ensures that the restriction of
\( \sqsubseteq^\star_A \) (for $A$ as in the hypothesis of Lemma
\ref{lemma:invuniv1}) to spaces in which every point realizes the distance \(
r \) is still invariantly universal.
\end{remark}

\begin{lemma} \label{lemma:invuniv2}
Let \( A = \{ 0 \} \cup \{ r_n \mid n \in \omega \} \) be well-spaced with \(
(r_n)_{n \in \omega} \) a strictly decreasing sequence converging to \( 0 \).
Then \( \V^\star_A \) is Borel in \( F(\mathbb{U}) \) and \(
\sqsubseteq^\star_A \) is invariantly universal.
\end{lemma}

\begin{proof}
The fact that under the hypotheses of the lemma the set \( \V^\star_A \) is Borel follows from
Theorem~\ref{proponlyultrametric}(2): since \( A \) is well-spaced, then \(
\V^\star_A = \U^\star_A \), and the latter is a Borel subset of \(
F(\mathbb{U}) \). The fact that \( \sqsubseteq^\star_A \) is invariantly
universal is proved in~\cite[Theorem 5.19]{cammarmot} for the special case \(
r_n = 2^{-n} \), but notice that the same proof works with any other choice
for the \( r_n \)'s as well.
\end{proof}

\begin{remark}\label{rmk:decreasing}
The proof of~\cite[Theorem 5.19]{cammarmot} actually shows that \(
\sqsubseteq^\star_A \) (for $A$ as in the hypothesis of Lemma
\ref{lemma:invuniv2}) is already invariantly universal when restricted to
spaces \( X \in \V^\star_A \) such that for all \( n \in \omega \) there is
\( m \in \omega \) with \( d_X(\psi_n(X),\psi_m(X)) = r_0 \).
\end{remark}

\begin{theorem}\label{thm:invuniv}
Let \( A \in \D \), and assume that \( A \) is either ill-founded or not
well-spaced. Then \( \sqsubseteq^\star_A \) is invariantly universal, meaning that:%
\footnote{Formally, the definition of invariant universality is given only
for quasi-orders with a standard Borel domain (Definition~\ref{def:cB}). Here
we are considering the natural generalization of this concept to quasi-orders
whose domain is an arbitrary Borel space.} For every analytic quasi-order \(
R \) on a standard Borel space there is a Borel subset \( C' \) of \(
F(\mathbb{U}) \) invariant under isometry such that $C' \subseteq \V^\star_A$
and \( R \sim_B {\sqsubseteq \restriction C'} \).
\end{theorem}

\begin{proof}
By Lemma~\ref{lemma:invuniv1} and Lemma~\ref{lemma:invuniv2}, we can assume
that \( A \) is neither of the form \( \{ 0,r,r' \} \) with \( r < r' \leq 2r
\), nor of the form \( \{ 0 \} \cup \{ r_n \mid n \in \omega \} \) with \(
(r_n)_{n \in \omega} \) a strictly decreasing sequence converging to \( 0 \)
such that \( 2r_{n+1} < r_n \).
\begin{claim} \label{claim:invuniv1}
There exist \( A' \subseteq A \), \( B \subseteq \V^\star_{A'} \), \( r_0 \in
A' \), and \(r_1 , \bar{r} \in A \) such that:
\begin{enumerate}
\item \( B \) is a Borel subset of \( F(\mathbb{U}) \) which is invariant
    under isometry;
\item \( \sqsubseteq \restriction B \) is invariantly universal;
\item \(r_0, r_1 < \bar{r}\), \(r_0 \neq r_1\), and \( r \leq \bar{r} \)
    for all \( r \in A' \);
\item for every \( X \in B \) and for every countable dense \( D \subseteq
    X \) it  holds
\[
\forall x \in X \, \exists y \in D \, (d_X(x,y) = r_0).
 \]
\end{enumerate}
\end{claim}

\begin{proof}
Suppose first that \( A \) is not well-spaced, and let \( r,r' \in A \) be
such that \( r < r' \leq 2r \). Set \( A' = \{ 0, r, r' \} \).  Set \( r_0 =
r \), and pick \( \bar{r} \in A \cap (r', + \infty) \) if such set is
nonempty, and \( \bar{r} = r' \) otherwise. Finally, let \( r_1 \) be any
element of \( A \) distinct from \( r_0  \) and smaller than \( \bar{r} \);
indeed, if \( \bar{r} \neq r' \) we can just take \( r_1 = r' \), otherwise
the existence of such an \( r_1 \) is guaranteed by the fact that \( \bar{r}
= r'  = \max A \) and \( A \neq A' \). These choices ensure that (3) is
satisfied.

Since $\V^\star_{A'}$ is Borel by Lemma \ref{lemma:invuniv1} the set \( B
\subseteq F(\mathbb{U}) \) consisting of the $X \in \V^\star_{A'}$ such that
\[
\forall n\, \exists m\, \forall i \leq n (\psi_m(X) \neq \psi_i(X)) \wedge
\forall n\, \exists m (d_X(\psi_n(X),\psi_m(X)) = r)
\]
is Borel. Notice that every \( X \in B \) is discrete and infinite, so that
the \( \psi_n \)'s enumerate the entire \( X \) and no distances outside $A'$
are possible. It easily follows that \( B \) is invariant under isometry so
that (1) is satisfied. Condition (4) is satisfied as well by the second part
of the definition of \( B \) and the fact that we set \( r_0 = r \). Finally,
condition (2) follows from Lemma~\ref{lemma:invuniv1} and
Remark~\ref{rmk:invuniv1}.

Assume now that \( A \) is well-spaced and ill-founded. Notice that any
strictly decreasing sequence \( (r_n)_{n \in \omega} \) in \( A \) must
converge to \( 0 \). Also we may
assume without loss of generality that there is \( \bar{r} \in A \) with \(
r_0 < \bar{r} \) (otherwise we shift the decreasing sequence by one). Then
setting \( A' = \{ 0 \} \cup \{ r_n \mid n \in \omega \} \) we get that (3)
is satisfied. Moreover, \( A' \) is well-spaced by case assumption, hence \(
\V^\star_{A'} \) is Borel by Lemma~\ref{lemma:invuniv2}.

Let \( B \) be the collection of those \( X \in \V^\star_{A'} \) such that
\[
\forall n \exists m (d_X(\psi_n(X),\psi_m(X)) = r_0).
 \]
The set \( B \) is clearly Borel in \( F(\mathbb{U}) \). We will now show
that (4) is satisfied for such \( B \): since condition (4) is preserved by
isometry, it will also follow that \( B \) is invariant under isometry, i.e.\
that (1) is satisfied. So let \( X \in B \), let \( D \) be dense in \( X \),
and let \( x \in X \) be arbitrary. Let \( n \in \omega \) be such that \(
d_X(\psi_n(X),x) < r_0 \), and let \( m \in \omega \) be such that \(
d_X(\psi_n(X),\psi_m(X)) = r_0 \), which exists because \( X \in B \). Then
since \( X \) is ultrametric by Theorem~\ref{proponlyultrametric}(2) and the
fact that \( A' \) is well-spaced, we also get \( d_X(x, \psi_m(X)) = r_0 \).
Using the density of \( D \), pick \( y \in D \) such that \(
d_X(\psi_m(X),y) < r_0 \): using again the fact that \( X \) is ultrametric,
we get \( d_X(x,y) = r_0 \), as required. Finally, part (2) follows from
Lemma~\ref{lemma:invuniv2} and Remark~\ref{rmk:decreasing}.
\end{proof}

\begin{claim} \label{claim:invuniv2}
Let \( A' \), \( B \subseteq \V^\star_{A'} \), \( r_0 \), \(r_1 \), and \(
\bar{r} \) be as in Claim~\ref{claim:invuniv1}. Then there is a Borel map \(
f \colon B \to \V^\star_A \) such that
\begin{enumerate}
\item \( C = [f(B)]_{\cong} = \{ Y \in F(\mathbb{U}) \mid \exists X \in B
    (f(X)  \cong Y) \} \) is Borel (and obviously invariant under
    isometry);
\item \( f \) reduces \( \sqsubseteq \restriction B \) to \(
    \sqsubseteq^\star_A  \) and \( \cong \restriction B \) to \(
    \cong^\star_A \);
\item there is a Borel \( g \colon C \to B  \) such that \( (g \circ f)(X)
    \cong X \) for all \( X \in B \) and \( (f \circ g)(Y) \cong Y \) for
    all \( Y \in C \).
\end{enumerate}
\end{claim}

\begin{proof}
Fix \( W \in \V^\star_{A \setminus \{ r_0,r_1 \bar{r} \}} \), and let \( Z =
W \oplus_{\bar{r}} V \), where \( V \) consists of two points \(z_0\) and \(
z_1 \) at distance \( r_1 \). Notice that by Fact \ref{fact:oplus} \( Z \in
\V^\star_{A \setminus \{ r_0 \}} \) and that \( z_0, z_1 \) are the unique
points of \(Z \) which realize the distance \( r_1 \). Moreover, they are
isolated in \( Z \), so they belong to any dense subset of \( Z \). Finally,
notice that \( d_Z(z_0,z) = d_Z(z_1,z) \) for any \( z \in W \): it follows
that if \( X \in B \subseteq \V^\star_{A'} \), the choice of the gluing point
in \( X \) (because of Claim~\ref{claim:invuniv1}(3)) and of \( z_0 \) or \(
z_1 \) as gluing point in $Z$ does not change the space \( X
\oplus_{\bar{r}} Z \).

Let now \( f \)  be the Borel map
\[
f \colon B \to \V^\star_A, \qquad X \mapsto \widetilde{X} = X \oplus_{\bar{r}} Z,
\]
where \( X \) and \( Z \) are glued using one of \( z_0, z_1 \). Notice that
\( \widetilde{X} \in \V^\star_A \) because \( r_0 \in A' \), \( X \in
\V^\star_{A'} \), and \( Z \in \V^\star_{A \setminus \{ r_0 \}} \), so that
\( f \) is well-defined.

It is not hard to see that \( f \) satisfies (2). Indeed, if \( \varphi \) is
an isometry (respectively, an isometric embedding) between \( X, X' \in B \),
then \( \varphi \cup \mathrm{id}_Z \) is an isometry (respectively, an
isometric embedding) between \( \widetilde{X} \) and \( \widetilde{X}' \).
Conversely, if \( \psi \) is an isometry (respectively, an isometric
embedding) between \( \widetilde{X} \) and \( \widetilde{X}' \), then \( \psi
\restriction X \) is an isometry (respectively, an isometric embedding)
between \( X \) and \( X' \) because of Claim~\ref{claim:invuniv1}(4) and the
fact that \( r_0 \notin D(Z) \).

We now check that \( C = [f(B)]_{\cong} \) is Borel in \( F(\mathbb{U}) \).
First observe that since topologically \( \widetilde{X} \) is  a direct sum
of \( X \) and \( Z \), then for any dense subset $D$ of $ \tilde X $ one has
$\forall x\in X\ \exists y\in D\ d_X(x,y)=r_0$. Together with the facts that
\( r_0 \notin D(Z) \) and $r_0 \neq \bar r$, this allows us to recover \( X
\) from the space \( \widetilde{X} \) as the completion (i.e.\ closure in \(
\mathbb{U} \)) of
\[
\{ \psi_n(\widetilde{X}) \mid \exists m (d(\psi_n(\widetilde{X}),\psi_m(\widetilde{X}))) = r_0 \},
\]
and \( Z \) as the completion of
\[
\{ \psi_n(\widetilde{X}) \mid \neg \exists m (d(\psi_n(\widetilde{X}),\psi_m(\widetilde{X}))) = r_0 \}.
\]

We now generalize this process to an arbitrary \( Y \in F(\mathbb{U}) \). Let
\( \mathrm{Rlz}(n,Y,r) \) be an abbreviation for the Borel condition
\[
\exists m (d(\psi_n(Y),\psi_m(Y)) = r),
 \]
and set
\[
X(Y) = \mathrm{cl}(\{ \psi_n(Y) \mid \mathrm{Rlz}(n,Y,r_0) \})
 \]
and
\[
Z(Y) = \mathrm{cl}(\{ \psi_n(Y) \mid \neg \mathrm{Rlz}(n,Y,r_0) \}).
 \]
Notice that the maps \( Y \mapsto X(Y) \) and \( Y \mapsto Z(Y) \) are Borel
functions from \( F(\mathbb{U}) \) into itself, and for \( X \in B \) we have
\( X(\widetilde{X}) \cong X \) and \( Z(\widetilde{X}) \cong Z \). Then for an
arbitrary \( Y \in F(\mathbb{U}) \) we have that \( Y \in C \) if and only if
\( \exists X \in B \, (Y  \cong \widetilde{X}) \), if and only if \( X(Y) \in
B \), \( Z(Y) \cong Z \), and \( Y \cong X(Y) \oplus_{\bar{r}} Z(Y) \) (where
the two spaces are glued via one of the only two points in \( Z(Y) \)
realizing \( r_1 \) in \( Z(Y) \)), if and only if
\begin{multline*}
X(Y) \in B \wedge Z(Y) \cong Z \wedge
\exists n [\psi_n(Y) \in Z(Y) \wedge \mathrm{Rlz}(n,Y,r_1) \wedge \\
\forall m,k (\psi_m(Y) \in X(Y) \wedge \psi_k(Y) \in Z(Y) \Rightarrow d(\psi_m(Y),\psi_k(Y)) = \max \{ \bar{r},d(\psi_k(Y),\psi_n(Y)) \})].
\end{multline*}
This is a Borel condition (here we are also using the fact that since
isometry on \( F(\mathbb{U}) \) is Borel bi-reducible with an orbit
equivalence relation, then the isometry class of a fixed element \( Z \) is
Borel), whence \( C \) is a Borel subset of \( F(\mathbb{U}) \).

Finally, let
\[
g \colon C \to B, \qquad Y \mapsto X(Y),
 \]
where \( X(Y) \) is as above. As already noticed, such map is Borel and \( (g
\circ f)(X) = g(\widetilde{X}) = X(\widetilde{X}) \cong X \) for every \( X
\in B \). Conversely, for every \( Y \in C \) we have \( (f \circ g)(Y) =
\widetilde{X(Y)} \cong X(Y) \oplus_{\bar{r}} Z \): since \( Y \cong X(Y)
\oplus_{\bar{r}} Z(Y) \) where \( Z(Y) \cong Z \) and \( X(Y) \) and \( Z(Y)
\) are glued using one of the only two points of \( Z(Y) \) which realize the
distance \( r_1 \) in \( Z(Y) \), it follows that \( (f \circ g)(Y) \cong
X(Y) \oplus_{\bar{r}} Z \cong X(Y) \oplus_{\bar{r}} Z(Y) \cong Y \).
\end{proof}

Notice that conditions (2) and (3) of Claim~\ref{claim:invuniv2} imply that
\( g \) reduces \( \sqsubseteq \restriction C \) and \( \cong \restriction C
\) to \( \sqsubseteq \restriction B \) and \( \cong \restriction B \),
respectively.

Let now \( R \) be an arbitrary analytic quasi-order on a standard Borel
space. By Claim~\ref{claim:invuniv1} there is \( B' \subseteq B \) Borel (in
both \( B \) and \( F(\mathbb{U}) \)) and invariant under isometry such that
\( R \sim_B {\sqsubseteq \restriction B'} \); let \( h_1 \colon \dom(R) \to
B' \) and \( h_2 \colon B' \to \dom(R) \) be witnesses of this last fact. Let
\( C' = [f(B')]_{\cong} = \{ Y \in F(\mathbb{U}) \mid \exists X \in B' (f(X)
\cong Y) \} \). Then \( C' \) is invariant under isometry by definition, and
it is a Borel subset of \( F(\mathbb{U}) \): indeed, this follows from the
fact that \( C'= g^{-1}(B') \) by Claim~\ref{claim:invuniv2}(3), and that \(
C \) is Borel in \( F(\mathbb{U}) \) by Claim~\ref{claim:invuniv2}(1).
Moreover, \( C' \subseteq C \subseteq \V^\star_A \). Finally, \( f \circ h_1
\) witnesses \( R \leq_B {\sqsubseteq \restriction C'} \), while \( h_2 \circ
(g \restriction C') \) witnesses \( {\sqsubseteq \restriction C'} \leq_B R \)
by parts (2) and (3) of Claim~\ref{claim:invuniv2}.
\end{proof}

\section{Open problems}

Theorem \ref{isomstarA} gives a fairly neat picture of the behaviour of the
relations $ \isom_A^\star$, except for the following case.

\begin{question}\label{qA}
If $A \in \D$ is dense in some right neighborhood of $0$ but does not contain
entirely any of them, can something more precise be said about the complexity
of $\isom_A^\star$?
\end{question}

Since for such $A$ we have that $\V^\star_A$ consists only of
zero-dimensional spaces by Theorem \ref{proponlyultrametric}(1), this
question is linked to one of the main questions of \cite{Gao2003} which is
still open (this question is discussed more at depth in \cite{ultrametric},
where it is labeled Question 7.1).

\begin{question}
What is the complexity of isometry between zero-dimensional Polish metric
spaces?
\end{question}

A step to get some insight into Question \ref{qA} would be to answer the
following:

\begin{question}
Which $A$ and $A'$ as in Question \ref{qA} are such that ${\isom_A^\star}
\leq_B {\isom_{A'}^\star}$?
\end{question}

This question has a positive answer for a given pair \( A,A' \) whenever
there exists $f \colon A \to A'$ which is injective and Polish metric
preserving. In fact, arguing as in the proof of case (iii) of
Theorem~\ref{thm:isomA} one can build a witness for ${\isom_A^\star} \leq_B
{\isom_{f(A)}^\star}$ by changing the distances of the spaces using $f$, and,
if $f(A) \subsetneq A'$, successively applying
Proposition~\ref{questionclemens} to obtain ${\isom_{f(A)}^\star} \leq_B
{\isom_{A'}^\star}$. For example, in this way one sees that ${\isom_{
\QQ^+}^\star} \leq_B {\isom_{(\RR^+ \setminus \QQ) \cup \{0\}}^\star}$ via
the map $f(q) = \sqrt2\, q$.

Another problem related to Question \ref{qA} is the following:

\begin{question}
Is it the case that ${\isom_A} \sim_B {\isom^\star_A}$ for every $A \in \D$?
\end{question}

Theorem \ref{thm:isomA} and the results following it give a positive answer
to this question under a wide range of hypotheses on $A$, but we do not know
if there are distance sets which do not satisfy any of those hypotheses.

\bibliographystyle{alpha}

\end{document}